\documentclass[12pt,twoside]{amsart}
\usepackage{amssymb,amsmath,amsthm, amscd, enumerate, mathrsfs}
\usepackage{graphicx, hhline}
\usepackage[all]{xy}
\title[Log surfaces]{Minimal model theory for log 
surfaces in Fujiki's class $\mathcal C$}
\author{Osamu Fujino}
\date{2020/1/20, version 0.08}
\address{Department of Mathematics, Graduate School of Science, 
Osaka University, Toyonaka, Osaka 560-0043, Japan}
\email{fujino@math.sci.osaka-u.ac.jp}

\DeclareMathOperator{\Supp}{Supp}
\DeclareMathOperator{\Exc}{Exc}
\DeclareMathOperator{\Pic}{Pic}
\DeclareMathOperator{\NE}{\overline{NE}}
\newtheorem{thm}{Theorem}[section]
\newtheorem{lem}[thm]{Lemma}

\newtheorem{conj}[thm]{Conjecture}
\newtheorem{cor}[thm]{Corollary}
\newtheorem{claim}{Claim}

\theoremstyle{definition}
\newtheorem{ex}[thm]{Example}
\newtheorem{defn}[thm]{Definition}
\newtheorem{rem}[thm]{Remark}
\newtheorem*{ack}{Acknowledgments}  
\newtheorem{step}{Step}

\newtheorem{question}[thm]{Question}

\makeatletter
    
    \@addtoreset{equation}{section}
\makeatother

\begin{document}

\begin{abstract}
We establish the minimal model theory for 
$\mathbb Q$-factorial log surfaces and log canonical surfaces 
in Fujiki's class $\mathcal C$. 
\end{abstract}

\maketitle 
\tableofcontents

\section{Introduction}\label{f-sec1}

A log surface $(X, \Delta)$ in Fujiki's class $\mathcal C$ 
consists of a compact normal analytic 
surface $X$ that is 
bimeromorphically equivalent to a compact 
K\"ahler manifold and a $\mathbb Q$-divisor $\Delta$ on $X$ whose 
coefficients are in $[0, 1]\cap \mathbb Q$ such that 
$K_X+\Delta$ is $\mathbb Q$-Cartier, that is, 
there exists a positive integer $m$ such that 
$m\Delta$ is integral and $\left( \omega^{\otimes m}_X\otimes 
\mathcal O_X(m\Delta)\right)^{**}$ is locally free, where 
$\omega_X$ is the canonical sheaf of $X$. In this paper, 
we establish the following theorem, which is 
a generalization of the minimal model theory for projective 
$\mathbb Q$-factorial log surfaces obtained in \cite{fujino-surfaces} 
(for some 
related topics, see \cite{fujino-tanaka}, \cite{tanaka}, 
\cite{hashizume1}, \cite{liu}, \cite{miyamoto}, and 
\cite[Section 4.10]{fujino-foundations}). 

\begin{thm}[Minimal model theory for $\mathbb Q$-factorial log 
surfaces in Fujiki's class $\mathcal C$]\label{f-thm1.1}
Let $(X, \Delta)$ be a $\mathbb Q$-factorial 
log surface in Fujiki's class $\mathcal C$. 
Then we can construct a finite sequence of 
projective bimeromorphic morphisms starting from $(X, \Delta)$: 
$$
(X, \Delta)=:(X_0, \Delta_0)\overset{\varphi_0}{\longrightarrow} 
(X_1, \Delta_1) \overset{\varphi_1}{\longrightarrow} 
\cdots 
\overset{\varphi_{k-1}}{\longrightarrow} 
(X_k, \Delta_k)=:(X^*, \Delta^*)
$$ 
such that $(X_i ,\Delta_i)$, where 
$\Delta_i:={\varphi_{i-1}}_*\Delta_{i-1}$, 
is a $\mathbb Q$-factorial log surface in Fujiki's class $\mathcal C$ 
and that 
$\Exc(\varphi_i)=:C_i\simeq \mathbb P^1$ and 
$-(K_{X_i}+\Delta_i)\cdot C_i>0$ for 
every $i$. The final model $(X^*, \Delta^*)$ satisfies one of the 
following conditions. 
\begin{itemize}
\item[(i)] {\em{(Good minimal model).}}~$K_{X^*}+\Delta^*$ is semi-ample. 
\item[(ii)] {\em{(Mori fiber space).}}~There exists a surjective morphism  
$g\colon X^*\to W$ onto a normal projective variety $W$ with 
connected fibers such that 
$-(K_{X^*}+\Delta^*)$ is $g$-ample, $\dim W<2$, 
and the relative Picard number $\rho(X^*/W)$ is one.  
\end{itemize}
We note that 
\begin{itemize}
\item[(1)] if $X_{i_0}$ is projective for some $i_0$ then 
$X_i$ is automatically projective for every $i$, and 
\item[(2)] if $X_{i_0}$ has only rational singularities for some 
$i_0$ then all the singularities of $X_i$ are rational for every $i$. 
\end{itemize}
We note that the above sequence of 
contraction morphisms 
is nothing but the minimal model 
program for projective $\mathbb Q$-factorial 
log surfaces established in \cite{fujino-surfaces} when $X$ 
is projective and that 
$X$ is automatically projective when $\kappa (X, K_X+\Delta)=-\infty$ or 
$2$. 
\end{thm}

Theorem \ref{f-thm1.1} is not difficult to check once we know the 
minimal model theory for projective 
$\mathbb Q$-factorial log surfaces in \cite{fujino-surfaces}, 
the Enriques--Kodaira classification of compact complex surfaces 
(see \cite[Chapter VI]{bhpv}), and some basic results 
on complex analytic spaces. 
We note that Theorem \ref{f-thm1.1} includes the abundance 
theorem for $\mathbb Q$-factorial log surfaces 
in Fujiki's class $\mathcal C$. 

\begin{thm}[Abundance theorem for $\mathbb Q$-factorial 
log surfaces in Fujiki's class $\mathcal C$, 
see Theorem \ref{h-thm7.2}]\label{f-thm1.2}
Let $(X, \Delta)$ be a $\mathbb Q$-factorial log surface 
in Fujiki's class $\mathcal C$. 
Assume that $(K_X+\Delta)\cdot C\geq 0$ for 
every curve $C$ on $X$. 
Then $K_X+\Delta$ is semi-ample. 
\end{thm}

From the minimal model theoretic viewpoint, it is very natural 
to treat log canonical 
surfaces $(X, \Delta)$ in Fujiki's class $\mathcal C$. 
Unfortunately, $X$ is not necessarily $\mathbb Q$-factorial in this case. 
So we can not directly apply Theorem \ref{f-thm1.1} 
to log canonical surfaces in Fujiki's class $\mathcal C$. 
In order to establish the minimal model theory 
for log canonical surfaces in Fujiki's class $\mathcal C$, 
we prove the following theorem. 

\begin{thm}[Projectivity of log canonical surfaces in Fujiki's 
class $\mathcal C$ with negative Kodaira dimension, 
see Theorem \ref{f-thm9.1}]\label{f-thm1.3}
Let $(X, \Delta)$ be a log canonical surface in Fujiki's class $\mathcal C$. 
Assume that $\kappa (X, K_X+\Delta)=-\infty$ holds. 
Then $X$ is projective. 
\end{thm}

The proof of Theorem \ref{f-thm1.3} is much more difficult 
than we expected. We prove it with the aid of 
the classification of two-dimensional log canonical singularities. 
We note that there are non-projective normal 
complete rational surfaces (see \cite[Section 4]{nagata}). 
Fortunately, such surfaces do not appear under the assumption of 
Theorem \ref{f-thm1.3}. 
Since Nagata's example in \cite[Section 4]{nagata} 
is not log canonical, we explicitly construct some examples 
of complete non-projective log canonical algebraic surfaces in 
Section \ref{p-sec12} for the reader's convenience. Our construction, 
which was suggested by Kento Fujita,  
is arguably 
simpler than Nagata's original 
and classical one (see \cite[Section 4]{nagata}).
Here, we explain the most interesting example. 

\begin{ex}[see Example \ref{p-ex12.3}]\label{f-ex1.4} 
There exists a complete non-projective 
log canonical algebraic surface $S$ with $\Pic (S)=\{0\}$ and 
$K_S\sim 0$. 
In particular, $\kappa (S, K_S)=0$ holds. 
\end{ex}

For the details of Example \ref{f-ex1.4} and 
some other examples of complete non-projective 
algebraic surfaces, see Section \ref{p-sec12}, 
where the reader can find some examples of complete non-projective 
log canonical algebraic surfaces $S$ with 
$\Pic(S)=\{0\}$, $\NE(S)=\mathbb R_{\geq 0}$, or $\NE(S)=N_1(S)$. 

\medskip 

Thus, by using Theorem \ref{f-thm1.3}, 
we have the following minimal model theory 
for log canonical surfaces in Fujiki's class $\mathcal C$. 

\begin{thm}[Minimal model theory for log canonical 
surfaces in Fujiki's class $\mathcal C$]\label{f-thm1.5}
Let $(X, \Delta)$ be a log canonical 
surface in Fujiki's class $\mathcal C$. 
Then we can construct a finite sequence of 
projective bimeromorphic morphisms starting from $(X, \Delta)$: 
$$
(X, \Delta)=:(X_0, \Delta_0)\overset{\varphi_0}{\longrightarrow} 
(X_1, \Delta_1) \overset{\varphi_1}{\longrightarrow} 
\cdots 
\overset{\varphi_{k-1}}{\longrightarrow} 
(X_k, \Delta_k)=:(X^*, \Delta^*)
$$ 
such that $(X_i ,\Delta_i)$, where 
$\Delta_i:={\varphi_{i-1}}_*\Delta_{i-1}$, 
is a log canonical surface in Fujiki's class $\mathcal C$ and that 
$\Exc(\varphi_i)=:C_i\simeq \mathbb P^1$ and 
$-(K_{X_i}+\Delta_i)\cdot C_i>0$ for 
every $i$. The final model $(X^*, \Delta^*)$ satisfies 
one of the following conditions. 
\begin{itemize}
\item[(i)] {\em{(Good minimal model).}}~$K_{X^*}+\Delta^*$ is semi-ample. 
\item[(ii)] {\em{(Mori fiber space).}}~There exists a surjective morphism  
$g\colon X^*\to W$ onto a normal projective variety $W$ with 
connected fibers such that 
$-(K_{X^*}+\Delta^*)$ is $g$-ample, $\dim W<2$, 
and the relative Picard number $\rho(X^*/W)$ is one.  
\end{itemize}
We note that 
\begin{itemize}
\item[(1)] if $X_{i_0}$ is projective for some $i_0$ then 
$X_i$ is automatically projective for every $i$, 
\item[(2)] if $X_{i_0}$ has only rational singularities for some 
$i_0$ then all the singularities of $X_i$ are rational for every $i$, and 
\item[(3)] if $X_{i_0}$ is $\mathbb Q$-factorial for some 
$i_0$ then so is $X_i$ for every $i$. 
\end{itemize}
We note that the above sequence of 
contraction morphisms 
is nothing but the usual minimal model 
program for projective log canonical 
surfaces $($see \cite{fujino-surfaces}$)$ when $X$ 
is projective and that $X$ is automatically projective 
when $\kappa(X, K_X+\Delta)=-\infty$ by 
Theorem \ref{f-thm1.3}. 
\end{thm}

In a series of papers (see \cite{horing-peternell1}, \cite{horing-peternell2}, 
and \cite{campana-horing-peternell}), Campana, H\"oring, and 
Peternell established the minimal model program 
and the abundance theorem for K\"ahler threefolds (see 
also \cite{horing-peternell3}). Their approach is essentially analytic. 
On the other hand, our approach is much more elementary than 
theirs and is not analytic. 
Although we mainly treat compact analytic surfaces in Fujiki's class 
$\mathcal C$, we do not discuss K\"ahler forms (or currents) 
on singular surfaces (see \cite{fujiki}). 

\medskip 

In Section \ref{f-sec11}, which is an appendix, we treat 
some 
vanishing theorems for proper bimeromorphic morphisms between 
analytic surfaces. They play an important role in this paper. 
Although they are more or less known to the experts, 
we explain the details for the reader's convenience 
because we can find no suitable references. 
We think that the results are useful for other applications. 
The most useful formulation is Theorem \ref{f-thm11.3} (2). 

\begin{thm}[see Theorem \ref{f-thm11.3}]\label{f-thm1.6}
Let $X$ be a normal analytic surface and 
let $\Delta$ be an effective $\mathbb Q$-divisor on $X$ 
whose coefficients are less than one such that 
$K_X+\Delta$ is $\mathbb Q$-Cartier. 
Let $f\colon X\to Y$ be 
a proper bimeromorphic morphism onto a normal 
analytic surface $Y$. 
Let $\mathcal L$ be a line bundle on $X$ and let $D$ 
be a $\mathbb Q$-Cartier Weil divisor on $X$. 
Assume that 
$\mathcal L\cdot C+(D-(K_X+\Delta))\cdot C\geq 0$ for every 
$f$-exceptional curve $C$ on $X$. 
Then $R^if_*(\mathcal L\otimes \mathcal O_X(D))=0$ holds 
for every $i>0$. 
\end{thm}

We explain the organization of this paper. 
In Section \ref{h-sec2}, we collect some basic definitions and 
results. 
In Section \ref{f-sec3}, 
we explain a very easy version of the basepoint-free theorem 
for projective bimeromorphic morphisms between surfaces 
(see Theorem \ref{f-thm3.11}). 
In Section \ref{f-sec4}, we collect some useful projectivity criteria for 
$\mathbb Q$-factorial compact analytic 
surfaces. In Section \ref{f-sec5}, we discuss the minimal model 
program for $\mathbb Q$-factorial log 
surfaces based on Sakai's 
contraction theorem, which is a slight generalization of 
Grauert's famous contraction theorem. 
Then we prove Theorem \ref{f-thm1.1} 
except for the semi-ampleness of 
$K_{X^*}+\Delta^*$. 
In Section \ref{f-sec6}, we briefly discuss the finite generation 
of log canonical rings of $\mathbb Q$-factorial log surfaces, 
which is essentially contained in \cite{fujino-surfaces}, and 
some related topics. 
In Section \ref{h-sec7}, we prove the non-vanishing theorem 
and 
the abundance theorem. Precisely speaking, 
we explain how to modify the arguments 
in \cite{fujino-surfaces} 
for $\mathbb Q$-factorial log surfaces in Fujiki's 
class $\mathcal C$. 
In Section \ref{f-sec8}, we discuss a contraction theorem for 
log canonical surfaces. 
A key point is that the exceptional curve automatically 
becomes $\mathbb Q$-Cartier. 
This simple fact plays a crucial role in our minimal model 
theory for log canonical surfaces. 
Section \ref{f-sec9} is devoted to the proof of 
Theorem \ref{f-thm1.3}, that is, the projectivity of log canonical 
surfaces in Fujiki's class $\mathcal C$ 
with negative Kodaira dimension. 
Our proof needs the classification 
of two-dimensional log canonical singularities. 
In Section \ref{f-sec10}, we prove Theorem \ref{f-thm1.5}, 
that is, the minimal model theory for log canonical surfaces 
in Fujiki's class $\mathcal C$. 
In Section \ref{f-sec11}, which is an appendix, 
we discuss some vanishing theorems 
for proper bimeromorphic morphisms 
between normal analytic surfaces. 
Fortunately, we need no deep analytic methods 
except for 
the theorem on formal functions for proper 
morphisms between analytic spaces. 
In Section \ref{p-sec12}, which is also 
an appendix, we construct some complete non-projective log canonical 
algebraic surfaces. 

\begin{ack}
The author was partially supported by 
JSPS KAKENHI Grant Numbers JP16H03925, JP16H06337.
He would like to thank Kenta Hashizume, Haidong Liu, 
and Hiromu Tanaka 
for very useful comments and pointing out some mistakes. 
He also would like to thank Kento Fujita very much for useful discussions and 
advice, and for allowing him to use his ideas on complete non-projective 
algebraic surfaces in Section \ref{p-sec12}. 
Finally, he thanks Seiko Hashimoto for her help and 
the referees for many comments. 
\end{ack}

We will use the minimal model theory for projective log surfaces 
defined over $\mathbb C$, the complex number field, 
established in \cite{fujino-surfaces}. 
We will freely use the basic notation of the minimal model 
theory as in \cite{fujino-fundamental} and \cite{fujino-foundations}. 

\section{Preliminaries}\label{h-sec2} 

In this section, we collect some basic definitions and 
results. 

\begin{defn}[Boundary and subboundary $\mathbb Q$-divisors]\label
{h-def2.1}
Let $X$ be an irreducible normal analytic space and 
let $\Delta$ be a $\mathbb Q$-divisor 
on $X$. 
If the coefficients of $\Delta$ are in $[0, 1]\cap \mathbb Q$ 
(resp.~$(-\infty, 1]\cap \mathbb Q$),  
then $\Delta$ is called a {\em{boundary}} (resp.~{\em{subboundary}}) 
$\mathbb Q$-divisor on $X$. 
\end{defn}

\begin{defn}[Operations for $\mathbb Q$-divisors]\label{h-def2.2}
Let $D$ be a $\mathbb Q$-divisor 
on a normal analytic space. 
Then $\lceil D\rceil$ (resp.~$\lfloor D\rfloor$) 
denotes the {\em{round-up}} (resp.~{\em{round-down}}) 
of $D$. 
We put $\{D\}:=D-\lfloor D\rfloor$ and 
call it the {\em{fractional part}} of $D$. 
\end{defn}

\begin{defn}[Algebraic dimensions]\label{h-def2.3}
Let $X$ be an irreducible compact analytic space. 
Let $\mathcal M (X)$ be the field of 
meromorphic functions on $X$. 
Then the transcendence degree of $\mathcal M (X)$ 
over $\mathbb C$ is called the {\em{algebraic dimension}} 
of $X$ and is denoted by $a(X)$. 
It is well known that $0\leq a(X)\leq \dim X$ holds. 
If $a(X)=\dim X$ holds, 
then we say that $X$ is {\em{Moishezon}}. 
We note that if $X$ is Moishezon then 
$X$ is an algebraic space 
which is proper over $\mathbb C$ (see \cite[Remark 3.7]{ueno}). 
\end{defn}

For the basic properties of $a(X)$, we recommend the 
reader to see \cite[Section 3]{ueno}. 

\begin{defn}\label{h-def2.4} 
Let $X$ be an irreducible compact normal analytic 
space such that 
$X$ is Moishezon. 
Then we can obtain the perfect pairing 
$$
N_1(X)\times N^1(X)\to \mathbb R
$$ 
induced from the intersection pairing of 
curves and line bundles as in the case where $X$ 
is projective. We note that 
$\rho (X):=\dim _{\mathbb R} N^1(X)<\infty$ always holds. 
We call $\rho(X)$ the {\em{Picard number}} of $X$. 
When $X$ is an algebraic variety, $\NE(X)$ ($\subset\!\! N_1(X)$) 
denotes the {\em{Kleiman--Mori cone}} 
of $X$. 
\end{defn}

In this paper, 
we do not consider $N_1(X)$ when $0\leq a(X)<\dim X$. 
When $X$ is a complete non-projective 
singular algebraic variety, 
$\NE(X)$ does not always behave well 
(see \cite{fujino-kleiman}, \cite{fujino-payne}, 
and Section \ref{p-sec12}).  

\begin{rem}\label{h-rem2.5} 
Let $X$ be a compact smooth analytic surface whose 
algebraic dimension $a(X)$ is zero. 
Then it is well known that there are 
only finitely many curves on $X$ (see Chapter IV, \cite[(8.2) 
Theorem]{bhpv}). 
\end{rem}

In the subsequent sections, we will repeatedly use the following 
well-known negativity lemma and its consequences 
without mentioning them explicitly. 
For the proof of Lemma \ref{h-lem2.6}, 
see \cite[Theorem 4-6-1]{matsuki}. 

\begin{lem}[Negativity lemma]\label{h-lem2.6}
Let $P\in Y$ be a germ of normal analytic surface and let 
$f:X\to Y$ be a proper 
bimeromorphic morphism 
from a smooth analytic surface $X$. 
Then $f^{-1}(P)$ is connected and has a negative definite 
intersection form. 
\end{lem}

Let us quickly recall the definition of the 
Iitaka dimension $\kappa$. 
For the details of $\kappa$, 
see \cite{nakayama} and \cite{ueno}. 

\begin{defn}[Iitaka dimensions]\label{h-def2.7} 
Let $X$ be an irreducible compact normal analytic space and let 
$\mathcal L$ be a line bundle on $X$. 
Then we set 
$$
\kappa (X, \mathcal L):=\limsup _{m\to \infty}
\frac{\log \dim _{\mathbb C} H^0(X, \mathcal L^{\otimes m})}
{\log m}
$$ 
and call it the {\em{Iitaka dimension}} of 
$\mathcal L$. 
It is well known that 
$$
\kappa (X, \mathcal L)\in \{-\infty, 0, 1, 2, \ldots, \dim X\} 
$$ 
holds. We can define $\kappa (X, D)$ for $\mathbb Q$-Cartier 
$\mathbb Q$-divisors $D$ on $X$ similarly. 
\end{defn}

We close this section with an easy lemma on rational singularities. 

\begin{lem}[{see \cite[Lemma 3.1]{fujino-rational}}]\label{h-lem2.8} 
Let $\varphi\colon X\to Y$ be a proper bimeromorphic 
morphism between normal analytic spaces. 
If $R^i\varphi_*\mathcal O_X=0$ for every $i>0$, 
then $X$ has only rational singularities if and only if so does 
$Y$. 
\end{lem}
\begin{proof}
The problem is local. 
So we can freely shrink $Y$ around an arbitrary given point. 
Let us consider a common resolution: 
$$
\xymatrix{
& W \ar[dr]^q\ar[dl]_p& \\ 
X \ar[rr]_\varphi& & Y. 
}
$$
By assumption and the Leray spectral sequence, 
we have $R^ip_*\mathcal O_W\simeq R^iq_*\mathcal O_W$ for 
every $i$. 
This implies the desired statement. 
\end{proof}

\section{Log surfaces}\label{f-sec3} 

In this section, we define $\mathbb Q$-factorial log 
surfaces and log canonical surfaces in Fujiki's class $\mathcal C$ 
and discuss a very easy version of the basepoint-free theorem for 
proper bimeromorphic morphisms between normal analytic surfaces. 

\medskip 

In this paper, we adopt the following definition of 
analytic spaces in Fujiki's class $\mathcal C$. 

\begin{defn}[Fujiki's class $\mathcal C$]\label{f-def3.1} 
Let $X$ be an irreducible 
compact analytic space. 
If $X$ is bimeromorphically equivalent to 
a compact K\"ahler manifold, 
then we say that $X$ is {\em{in Fujiki's class $\mathcal C$}}. 
\end{defn}

\begin{rem}\label{f-rem3.2}
Let $X$ be an irreducible compact analytic space. 
We note that if $X$ is Moishezon then $X$ is automatically 
in Fujiki's class $\mathcal C$. 
\end{rem}

We have a useful characterization of surfaces in 
Fujiki's class $\mathcal C$. 

\begin{lem}\label{f-lem3.3}
Let $X$ be an irreducible 
compact normal analytic surface. 
Then $X$ is in Fujiki's 
class $\mathcal C$ if and only if 
there exists a resolution of 
singularities $f\colon Y\to X$ such that 
$Y$ is K\"ahler, 
that is, $Y$ is a two-dimensional 
compact K\"ahler manifold and $f$ is a bimeromorphic morphism. 
\end{lem}
\begin{proof}
Note that a compact smooth analytic surface $S$ is K\"ahler 
if and only if 
the first Betti number $b_1(S)$ is even (see 
\cite[Chapter IV, (3.1) Theorem]{bhpv}). 
We also note that 
the first Betti number is preserved under blow-ups. 
Thus we can easily check the statement. 
\end{proof}

As an easy consequence of Lemma \ref{f-lem3.3} and its proof, 
we have: 

\begin{cor}\label{f-cor3.4}
Let $X$ be a compact 
normal analytic surface in Fujiki's class $\mathcal C$. 
Let $f\colon Y\to X$ be any resolution of singularities. 
Then $Y$ is a compact K\"ahler manifold. 
\end{cor}

Let us define canonical sheaves. 

\begin{defn}[Canonical sheaves]\label{f-def3.5}
Let $X$ be a normal analytic surface and 
let $\mathrm{Sing} X$ denote the 
singular locus of $X$. 
Then we have $\mathrm{codim}_X\mathrm{Sing}X\geq 2$ since 
$X$ is normal. 
Let $\omega_U$ be the canonical bundle 
of $U:=X\setminus \mathrm{Sing}X$. 
We put $\omega_X:=\iota_*\mathcal \omega_U$, 
where $\iota\colon U\hookrightarrow X$ is the natural 
open immersion, and call $\omega_X$ the {\em{canonical 
sheaf}} of $X$. 
\end{defn}

\begin{rem}\label{f-rem3.6}
Some normal analytic surface $X$ does not 
admit any non-zero meromorphic section 
of $\omega_X$. 
However, if there is no risk of confusion, 
we use the symbol $K_X$ as a formal divisor class with 
an isomorphism $\mathcal O_X(K_X)\simeq 
\omega_X$ and call it the {\em{canonical divisor}} of 
$X$. 
\end{rem}

In this paper, we adopt the following definition of 
log surfaces. 

\begin{defn}[Log surfaces]\label{f-def3.7} 
Let $X$ be an irreducible compact 
normal analytic surface and let $\Delta$ be a boundary 
$\mathbb Q$-divisor on $X$. 
Assume that $K_X+\Delta$ is $\mathbb Q$-Cartier. 
Note that this means that 
there exists a positive integer $m$ such that 
$m\Delta$ is integral and 
that $\left(\omega^{\otimes m}_X\otimes 
\mathcal O_X(m\Delta)\right)^{**}$ is locally free. 
Then the pair $(X, \Delta)$ is called a {\em{log surface}}. 
We say that a log surface $(X, \Delta)$ is {\em{in Fujiki's 
class $\mathcal C$}} when $X$ is in Fujiki's class $\mathcal C$. 
Let $(X, \Delta)$ be a log surface. 
Then we usually call $\kappa (X, K_X+\Delta)$ 
the {\em{Kodaira dimension}} of $(X, \Delta)$. 
\end{defn}

We need to define log canonical surfaces. 

\begin{defn}[Log canonical surfaces]\label{f-def3.8}
Let $(X, \Delta)$ be a log surface and let $f\colon Y\to X$ 
be a proper bimeromorphic morphism 
from a smooth analytic surface $Y$. 
Then we can write 
$
K_Y+\Delta_Y=f^*(K_X+\Delta)
$ 
with $f_*\Delta_Y=\Delta$. 
If the coefficients of $\Delta_Y$ are less than or equal to one 
for every $f\colon Y\to X$, 
then $(X, \Delta)$ is called a {\em{log canonical surface}}. 
\end{defn}

The notion of $\mathbb Q$-factoriality plays a crucial 
role in this paper. 

\begin{defn}[$\mathbb Q$-factoriality]\label{f-def3.9}
Let $X$ be an irreducible compact normal analytic surface and let 
$D$ be a $\mathbb Q$-divisor 
on $X$. 
Then we say that $D$ is {\em{$\mathbb Q$-Cartier}} if 
there exists a positive integer $m$ such that 
$mD$ is Cartier. 
If every Weil divisor on $X$ is $\mathbb Q$-Cartier, 
then we say that $X$ is {\em{$\mathbb Q$-factorial}}.  
\end{defn}

Lemma \ref{f-lem3.10} is well known. 

\begin{lem}\label{f-lem3.10}
Let $X$ be an irreducible compact normal analytic surface. 
Assume that $X$ has only rational singularities. 
Then $X$ is $\mathbb Q$-factorial. 
\end{lem}

\begin{proof}
This follows from \cite[Chapter II, 2.12.~Lemma]{nakayama}. 
\end{proof}

We close this section with a very easy version of 
the basepoint-free theorem for projective 
bimeromorphic morphisms between normal analytic 
surfaces (see also Remark \ref{f-rem11.4} below). 

\begin{thm}\label{f-thm3.11} 
Let $(X, \Delta)$ be a log surface and 
let $\varphi\colon X\to Y$ be a projective bimeromorphic morphism 
onto a normal analytic surface $Y$. 
Assume that $C:=\Exc(\varphi)$ is $\mathbb Q$-Cartier, 
$C\simeq \mathbb P^1$, and 
$-C$ and $-(K_X+\Delta)$ are $\varphi$-ample. 
Let $\mathcal L$ be a line bundle on $X$ with $\mathcal L\cdot C=0$. 
Then there exists a line bundle $\mathcal L_Y$ on $Y$ 
such that $\mathcal L\simeq \varphi^*\mathcal L_Y$ holds. 
In particular, $X$ is $\mathbb Q$-factorial 
if and only if so is $Y$. 
\end{thm}

\begin{proof}
In Step \ref{a-step1}, 
we will prove the existence of $\mathcal L_Y$. 
In Steps \ref{a-step2} and \ref{a-step3}, we will see that 
$X$ is $\mathbb Q$-factorial if and only if so is $Y$. 

\begin{step}\label{a-step1}
Since $C\simeq \mathbb P^1$ and $\mathcal L\cdot C=0$, 
$\mathcal L|_C$ is trivial. 
Since $-C$ and $-(K_X+\Delta)$ are $\varphi$-ample, 
we may assume that $C\leq \Delta$ by increasing 
the coefficient of $C$ in $\Delta$. 
Let us consider the following short exact sequence: 
$$
0\to \mathcal O_X(-C)\to \mathcal O_X\to \mathcal O_C\to 0. 
$$ 
Note that 
$$
\mathcal L\cdot C+\left(-C-(K_X+\Delta-C)\right)\cdot C
=\mathcal L\cdot C-(K_X+\Delta)\cdot C>0
$$ 
holds. Therefore, by Theorem \ref{f-thm11.3} below, 
we get $R^1\varphi_*(\mathcal L\otimes 
\mathcal O_X(-C))=0$. 
Thus, we have the following 
short exact sequence: 
\begin{equation}\label{eq-1}
0\to \varphi_*(\mathcal L\otimes \mathcal O_X(-C))\to \varphi_*\mathcal 
L \to \varphi_*(\mathcal L|_C)\to 0. 
\end{equation}
We note that 
$$
\varphi_*(\mathcal L|_C)=H^0(C, \mathcal L|_C)\simeq 
H^0(\mathbb P^1, \mathcal O_{\mathbb P^1}). 
$$ 
By \eqref{eq-1}, $\mathcal L$ is $\varphi$-free 
since $\mathcal L|_C$ is trivial. 
We note that $\varphi_*\mathcal O_X\simeq \mathcal O_Y$ since 
$Y$ is normal and $\varphi$ has connected fibers. 
Thus $\mathcal L_Y:=\varphi_*\mathcal L$ is a line 
bundle on $Y$ such that 
$\mathcal L\simeq \varphi^*\mathcal L_Y$ holds. 
\end{step}
\begin{step}\label{a-step2}
Assume that $X$ is $\mathbb Q$-factorial. 
We take a prime divisor $D$ on $Y$. 
Let $D'$ be the strict transform of $D$ on $X$. 
Then we can take $a\in \mathbb Q$ and a divisible 
positive integer $m$ such that $m(D'+aC)$ is Cartier 
and $m(D'+aC)\cdot C=0$. 
We put $\mathcal L=\mathcal O_X(m(D'+aC))$ and 
apply the result obtained in Step \ref{a-step1} 
to $\mathcal L$. 
Then $mD=\varphi_*(m(D'+aC))$ is Cartier. 
This means that $Y$ is $\mathbb Q$-factorial.  
\end{step}
\begin{step}\label{a-step3}
Assume that $Y$ is $\mathbb Q$-factorial. 
We take a prime divisor $D$ on $X$. 
Then $D':=\varphi_*D$ is a $\mathbb Q$-Cartier prime divisor on $Y$. 
Since $D=\varphi^*D'-aC$ holds for some 
$a\in \mathbb Q$, $D$ is $\mathbb Q$-Cartier. 
Therefore, $X$ is $\mathbb Q$-factorial. 
\end{step}
We complete the proof of Theorem \ref{f-thm3.11}. 
\end{proof}

\section{Projectivity criteria}\label{f-sec4}

Let us start with an easy but very useful projectivity criterion. 

\begin{lem}[Projectivity of $\mathbb Q$-factorial 
compact analytic surfaces]\label{f-lem4.1}
Let $X$ be a $\mathbb Q$-factorial 
compact analytic surface. 
Assume that the algebraic dimension $a(X)$ of $X$ is 
two,  
that is, $X$ is Moishezon. Then $X$ is projective. 
\end{lem}
\begin{proof}
By the assumption $a(X)=2$, 
we can construct a proper 
bimeromorphic morphism $f\colon Y\to X$ from a 
smooth projective surface $Y$. 
By the assumption $a(X)=2$ again, 
$X$ is an algebraic space which is proper over $\mathbb C$ 
by Artin's GAGA (see \cite[Remark 3.7]{ueno}). 
We take a very ample effective Cartier divisor $H$ on $Y$. 
We put $A=f_*H$. 
Since $X$ is $\mathbb Q$-factorial, 
$A$ is a $\mathbb Q$-Cartier divisor. 
Then we have $A\cdot C=H\cdot f^*C>0$ for every 
curve $C$ on $X$. In particular, we have $A^2>0$. 
Therefore, $A$ is ample by 
Nakai--Moishezon's ampleness criterion for algebraic spaces 
(see \cite[(1.4) Theorem]{pascul}).  
This implies that 
$X$ is projective.  
\end{proof}

The following corollary is obvious by Lemma \ref{f-lem4.1}. 

\begin{cor}\label{f-cor4.2}
Let $X$ be a $\mathbb Q$-factorial 
compact analytic surface. Assume that 
there exists a line bundle $\mathcal L$ such that 
$\kappa (X, \mathcal L)=2$, that is, $\mathcal L$ is a 
big line bundle. Then $X$ is projective. 
\end{cor}

\begin{proof}
By the assumption that $\mathcal L$ is big, 
we see that the algebraic dimension $a(X)$ of $X$ is two. 
Therefore, $X$ is projective by Lemma \ref{f-lem4.1}.  
\end{proof}

By Corollary \ref{f-cor4.2}, the minimal model theory  
for projective $\mathbb Q$-factorial log surfaces 
established in \cite{fujino-surfaces} works for 
$(X, \Delta)$ with $\kappa (X, K_X+\Delta)=2$ in 
Theorem \ref{f-thm1.1}. 

\medskip

By combining Lemma \ref{f-lem4.1} with the Enriques--Kodaira 
classification, 
we obtain the following projectivity criterion. 

\begin{lem}\label{f-lem4.3}
Let $(X, \Delta)$ be a $\mathbb Q$-factorial 
log surface in 
Fujiki's class $\mathcal C$ with 
$\kappa (X, K_X+\Delta)=-\infty$. 
Then $X$ is projective. 
\end{lem}
\begin{proof}
We take the minimal resolution $f\colon Y\to X$. 
We put $K_Y+\Delta_Y:=f^*(K_X+\Delta)$. 
Then we see that $\Delta_Y$ is effective by the negativity lemma 
and that $\kappa (Y, K_X+\Delta_Y)=\kappa (X, K_X+\Delta)=-\infty$ 
holds. 
Therefore, we obtain $\kappa (Y, K_Y)=-\infty$ by 
$\kappa (Y, K_Y)\leq \kappa (Y, K_Y+\Delta_Y)=-\infty$. 
Since $X$ is in Fujiki's class $\mathcal C$, 
the first Betti number $b_1(Y)$ of $Y$ is even. 
Therefore, by the Enriques--Kodaira classification 
(see \cite[Chapter VI]{bhpv}), 
$Y$ is a smooth projective 
surface. 
Thus, by Lemma \ref{f-lem4.1}, 
$X$ is projective. 
\end{proof}

We will repeatedly use the above projectivity criteria 
throughout this paper. 

\medskip 

We note that the statement of Theorem \ref{f-thm1.3} 
looks very similar to that of Lemma \ref{f-lem4.3}. 
However, a log canonical surface is not necessarily 
$\mathbb Q$-factorial. 
Therefore, Theorem \ref{f-thm1.3} 
is much harder to prove than Lemma \ref{f-lem4.3} 
(see Section \ref{f-sec9}). 

\section{Minimal model program for 
$\mathbb Q$-factorial log 
surfaces}\label{f-sec5}

By repeatedly using  
Grauert's contraction theorem, we can easily run a 
kind of the minimal model program for $\mathbb Q$-factorial 
log surfaces $(X, \Delta)$. 
We note that 
$X$ is not assumed to be in Fujiki's class $\mathcal C$ 
in Theorem \ref{f-thm5.1}. 
A key point of Theorem \ref{f-thm5.1} is the assumption that 
$X$ is $\mathbb Q$-factorial. 

\begin{thm}\label{f-thm5.1}
Let $(X, \Delta)$ be a compact $\mathbb Q$-factorial 
log surface. 
We assume that 
$\kappa (X, K_X+\Delta)\geq 0$. 
Then we can 
construct a finite sequence of 
projective bimeromorphic morphisms 
$$
(X, \Delta)=:(X_0, \Delta_0)\overset{\varphi_0}{\longrightarrow} 
(X_1, \Delta_1) \overset{\varphi_1}{\longrightarrow} 
\cdots 
\overset{\varphi_{k-1}}{\longrightarrow} 
(X_k, \Delta_k)=:(X^*, \Delta^*)
$$
with $\Delta_i:={\varphi_{i-1}}_*
\Delta_{i-1}$, $\Exc(\varphi_i)=:C_i\simeq \mathbb P^1$, 
and 
$-(K_{X_i}+\Delta_i)\cdot C_i>0$ for 
every $i$ such that $(K_{X^*}+\Delta^*)\cdot C\geq 0$ for 
every curve $C$ on $X^*$. We note that 
$(X_i, \Delta_i)$ is a compact $\mathbb Q$-factorial 
log surface for 
every $i$. 
\end{thm}

\begin{proof}
Since $\kappa (X, K_X+\Delta)\geq 0$, 
we can take an effective 
Cartier divisor $D\in |m(K_X+\Delta)|$ for some 
large and divisible positive integer $m$. 
If $m(K_X+\Delta)\cdot C=D\cdot C\geq 0$ for every curve 
$C$ on $X$, then we set $(X^*, \Delta^*):=(X_0, \Delta_0)=(X, \Delta)$. 
So we assume that 
there exists some irreducible curve $C$ on $X$ such that 
$D\cdot C<0$. 
Then $C$ is an irreducible component of $\Supp D$ and $C^2<0$. 
By Sakai's contraction theorem 
(see \cite[Theorem (1.2)]{sakai}), 
which is a slight generalization of Grauert's famous 
contraction theorem, 
we get a bimeromorphic morphism 
$\varphi_0\colon X=X_0\to X_1$ that contracts $C$ to a normal point 
of $X_1$. We take a divisible positive integer $l$ such that 
$lC$ is Cartier. 
Then $\mathcal O_X(-lC)$ is a $\varphi_0$-ample 
line bundle on $X$. 
In particular, $\varphi_0$ is a projective 
morphism. By construction, $-(K_X+\Delta)\cdot C>0$. 
Therefore, $-(K_X+\Delta)$ is $\varphi_0$-ample. 
Thus, $R^i{\varphi_0}_*\mathcal O_X=0$ for 
every $i>0$ by 
Theorem \ref{f-thm11.3} below. 
\begin{claim}\label{g-claim1}
$C$ is isomorphic to $\mathbb P^1$. 
\end{claim}
\begin{proof}[Proof of Claim \ref{g-claim1}]
We consider the following exact sequence: 
$$
\cdots \to R^1{\varphi_0}_*\mathcal O_X\to R^1{\varphi_0}_*
\mathcal O_C\to R^2{\varphi_0}_*\mathcal I_C\to \cdots, 
$$ 
where $\mathcal I_C$ is the defining ideal sheaf of $C$ on $X$. 
As we saw above, $R^1{\varphi_0}_*\mathcal O_X=0$ holds. 
Since $C$ is a curve, 
$R^2{\varphi_0}_*\mathcal I_C=0$ 
holds by the theorem on formal functions for 
proper morphisms between analytic spaces 
(see \cite[Chapter VI, Corollary 4.7]{banica-s}). 
Thus we get $H^1(C, \mathcal O_C)=R^1{\varphi_0}_*\mathcal O_C=0$ 
by the above exact sequence. 
This implies that 
$C$ is isomorphic to $\mathbb P^1$. 
\end{proof}

Therefore, by Theorem \ref{f-thm3.11}, 
we obtain that 
$(X_1, \Delta_1)$ is a $\mathbb Q$-factorial log surface. 
Since $\Supp D$ has only finitely many irreducible 
components, we get a desired sequence of contraction morphisms 
and finally obtain $(X^*, \Delta^*)$ with $(K_{X^*}
+\Delta^*)\cdot C\geq 0$ for every curve $C$ on $X^*$. 
\end{proof}

We note the following well-known lemma on extremal rays of  
projective surfaces. 

\begin{lem}\label{f-lem5.2}
Let $X$ be a normal projective surface and let 
$C$ be a $\mathbb Q$-Cartier irreducible 
curve on $X$ with $C^2<0$. 
Then the numerical equivalence class $[C]$ of $C$ spans an extremal 
ray of the Kleiman--Mori cone $\NE(X)$ of $X$. 
\end{lem}
\begin{proof}
This is obvious. For the proof, see \cite[Lemma 1.22]{kollar-mori}. 
\end{proof}

By Lemma \ref{f-lem5.2}, if $X$ is projective in Theorem 
\ref{f-thm5.1}, then the minimal model program 
in Theorem \ref{f-thm5.1} is nothing but the 
minimal model program for 
projective $\mathbb Q$-factorial log surfaces formulated and 
established in \cite{fujino-surfaces}. 
We also note that $X$ is projective in 
Theorem \ref{f-thm5.1} if the algebraic 
dimension $a(X)$ of $X$ is two by Lemma \ref{f-lem4.1}. 

\medskip 

We recall that $\mathbb Q$-factorial log surfaces 
$(X, \Delta)$ 
in Fujiki's class $\mathcal C$ with $\kappa(X, K_X+\Delta)=-\infty$ are 
projective by Lemma \ref{f-lem4.3}.  

\medskip 

Let us prove Theorem \ref{f-thm1.1} except for 
the semi-ampleness of $K_{X^*}+\Delta^*$. 

\begin{proof}[Proof of Theorem \ref{f-thm1.1}]
If $(K_X+\Delta)\cdot C\geq 0$ for every 
curve $C$ on $X$, then we put $(X^*, \Delta^*):=(X, \Delta)$. 
We will see that $K_{X^*}+\Delta^*$ is semi-ample in 
Theorem \ref{h-thm7.2}. 
We note that $X$ is in Fujiki's class $\mathcal C$. 
If $\kappa (X, K_X+\Delta)=-\infty$, then 
$X$ is projective by Lemma \ref{f-lem4.3}. 
Therefore, we can run the minimal model program for $\mathbb Q$-factorial 
projective log surfaces in \cite{fujino-surfaces} and finally get a Mori fiber space. 
Therefore, we may further assume that 
$\kappa (X, K_X+\Delta)\geq 0$. 
Then we can apply Theorem \ref{f-thm5.1} and 
finally get a model $(X^*, \Delta^*)$ 
such that 
$(K_{X^*}+\Delta^*)\cdot C\geq 0$ for every 
curve $C$ on $X^*$. 
In this case, by the abundance 
theorem:~Theorem \ref{h-thm7.2}, 
we will see that $K_{X^*}+\Delta^*$ is 
semi-ample. 

Since we have $R^1{\varphi_i}_*\mathcal O_{X_i}=0$ 
(see Theorem \ref{f-thm11.3}), 
$X_i$ has only rational singularities if and only if so does 
$X_{i+1}$  
by Lemma \ref{h-lem2.8}. 
Thus we have (2). 

Since each contraction $\varphi_i$ is projective, 
$X_i$ is projective when so is $X_{i+1}$. 
On the other hand, if $X_i$ is projective then so is 
$X_{i+1}$ 
because $\varphi_i$ is nothing but the usual contraction morphism 
associated to a $(K_{X_i}+\Delta_i)$-negative 
extremal ray (see Lemma \ref{f-lem5.2}). 
Thus, we have (1). 
\end{proof}

We obtained Theorem \ref{f-thm1.1} 
except for the semi-ampleness of $K_{X^*}+\Delta^*$, 
which will be proved in Section \ref{h-sec7}. 

\section{Finite generation of log canonical rings}\label{f-sec6}

In this section, we briefly discuss 
the finite generation of log canonical rings of pairs 
for the reader's convenience. 

\medskip 

The following theorem is the main result of this section, 
which is essentially contained in \cite{fujino-surfaces}. 

\begin{thm}[Finite generation of log canonical rings]
\label{f-thm6.1} 
Let $(X, \Delta)$ be a compact 
$\mathbb Q$-factorial log surface. 
Then the log canonical ring 
$$
\bigoplus _{m\geq 0} H^0(X, \mathcal O_X(\lfloor m(K_X+\Delta)\rfloor))
$$ 
is a finitely generated $\mathbb C$-algebra. 
We note that the sheaf 
$\mathcal O_X(\lfloor m(K_X+\Delta)\rfloor)$ denotes 
$\left(\omega^{\otimes m}_X\otimes \mathcal O_X(\lfloor 
m\Delta\rfloor)\right)^{**}$. 
\end{thm}

As an easy consequence of Theorem \ref{f-thm6.1}, we have: 

\begin{cor}\label{f-cor6.2} 
Let $(X, \Delta)$ be a compact log canonical surface. 
Then the log canonical ring 
$$
\bigoplus _{m\geq 0} H^0(X, \mathcal O_X(\lfloor m(K_X+\Delta)\rfloor))
$$ 
is a finitely generated $\mathbb C$-algebra. 
\end{cor}

We note that $X$ is not assumed to be in Fujiki's class $\mathcal C$ 
in Theorem \ref{f-thm6.1} and Corollary \ref{f-cor6.2}. 
 
\begin{proof}[Proof of Corollary \ref{f-cor6.2}]
Let $f\colon Y\to X$ be the minimal resolution. 
We put $K_Y+\Delta_Y:=f^*(K_X+\Delta)$. 
Since $(X, \Delta)$ is log canonical, 
we see that $\Delta_Y$ is a boundary $\mathbb Q$-divisor 
by the negativity lemma. 
By Theorem \ref{f-thm6.1}, 
the log canonical ring of 
$(Y, \Delta_Y)$ is a finitely generated 
$\mathbb C$-algebra. 
This implies that 
the log canonical ring of $(X, \Delta)$ is a finitely generated 
$\mathbb C$-algebra. 
\end{proof} 

Before we prove Theorem \ref{f-thm6.1}, 
let us recall the following easy well-known lemma for 
the reader's convenience. 

\begin{lem}\label{f-lem6.3}
Let $X$ be an irreducible compact normal analytic space and 
let $\mathcal L$ be a line bundle on $X$ 
such that $\kappa (X, \mathcal L)\leq 1$. 
Then the graded ring 
$$
R(X, \mathcal L):=\bigoplus _{m\geq 0} H^0(X, \mathcal 
L^{\otimes m})
$$ 
is a finitely generated $\mathbb C$-algebra. 
\end{lem}
\begin{proof}[Sketch of Proof]
If $\kappa (X, \mathcal L)=-\infty$ or $0$, 
then it is very easy to see that 
$R(X, \mathcal L)$ is a finitely generated 
$\mathbb C$-algebra. 
If $\kappa (X, \mathcal L)=1$, 
then we can 
reduce the problem to the case where $X$ is a smooth 
projective 
curve and $\mathcal L$ is an ample 
line bundle on $X$ by taking the Iitaka fibration 
(see \cite[(1.12) Theorem]{mori}). 
Thus, $R(X, \mathcal L)$ is a finitely generated $\mathbb C$-algebra 
when $\kappa(X, \mathcal L)\leq 1$. 
\end{proof}

Let us prove Theorem \ref{f-thm6.1}. 

\begin{proof}[Proof of Theorem \ref{f-thm6.1}]
By Lemma \ref{f-lem6.3}, 
we may assume that 
$\kappa (X, K_X+\Delta)=2$. 
Then, by Corollary \ref{f-cor4.2}, 
$X$ is projective. 
In this case, the log canonical 
ring 
$$
\bigoplus _{m\geq 0} H^0(X, 
\mathcal O_X(\lfloor m(K_X+\Delta)\rfloor))
$$ 
of $(X, \Delta)$ is a finitely generated 
$\mathbb C$-algebra by the minimal 
model theory for projective 
$\mathbb Q$-factorial 
log surfaces established in \cite{fujino-surfaces}. 
\end{proof}

Let us quickly see some results and conjectures on log canonical 
rings of higher-dimensional 
pairs. 

\begin{thm}[{\cite{bchm}, \cite{fujino-mori}, and 
\cite[Theorem 1.8]{fujino-some}}]\label{f-thm6.4}
Let $(X, \Delta)$ be a kawamata log terminal pair such that 
$\Delta$ is a $\mathbb Q$-divisor 
on $X$ and that 
$X$ is in Fujiki's class $\mathcal C$. 
Then the log canonical ring 
$$
\bigoplus _{m\geq 0} H^0(X, \mathcal O_X(\lfloor m(K_X+\Delta)\rfloor))
$$ 
is a finitely generated $\mathbb C$-algebra. 
\end{thm}

\begin{conj}\label{f-conj6.5}
Let $(X, \Delta)$ be a log canonical 
pair such that $\Delta$ is a $\mathbb Q$-divisor on $X$ 
and that $X$ is in Fujiki's class $\mathcal C$. 
Then the log canonical ring 
$$
\bigoplus _{m\geq 0} H^0(X, \mathcal O_X(\lfloor m(K_X+\Delta)\rfloor))
$$ 
is a finitely generated $\mathbb C$-algebra. 
\end{conj}

Conjecture \ref{f-conj6.5} is still widely open even when $X$ is projective 
(see \cite{fujino-finite}, \cite{fujino-some}, 
\cite{fujino-gongyo}, \cite{hashizume}, and 
\cite{fujino-liu}). 
When $X$ is projective in Conjecture \ref{f-conj6.5}, 
it is essentially equivalent to the existence problem of 
good minimal models for lower-dimensional 
varieties (for the details, see \cite{fujino-gongyo}). 
Note that Corollary \ref{f-cor6.2} completely settled Conjecture 
\ref{f-conj6.5} in dimension two. 

\medskip 

We close this section with a naive question. 

\begin{question}\label{f-que6.6} 
Let $X$ be an irreducible 
compact normal analytic surface such that 
$K_X$ is $\mathbb Q$-Cartier. 
Then is the canonical ring 
$$
\bigoplus _{m\geq 0} H^0(X, \mathcal O_X(mK_X))
$$ 
a finitely generated $\mathbb C$-algebra? 
\end{question}

We do not know the answer even when $X$ is projective. 

\section{Abundance theorem}\label{h-sec7} 

In this section, we prove the abundance theorem 
for $\mathbb Q$-factorial log surfaces in Fujiki's class $\mathcal C$. 

\medskip 

Let us start with the non-vanishing theorem. 

\begin{thm}[Non-vanishing theorem]\label{h-thm7.1}
Let $(X, \Delta)$ be a $\mathbb Q$-factorial log surface in Fujiki's class 
$\mathcal C$. 
Assume that 
$(K_X+\Delta)\cdot C\geq 0$ for every curve 
$C$ on $X$. 
Then we have $\kappa (X, K_X+\Delta)\geq 0$. 
\end{thm}
\begin{proof}
Let $f\colon Y\to X$ be the minimal resolution. 
We put $K_Y+\Delta_Y:=f^*(K_X+\Delta)$. 
Then $\Delta_Y$ is an effective $\mathbb Q$-divisor 
by the negativity lemma. 
If $\kappa (Y, K_Y)\geq 0$, 
then we have $$
\kappa (X, K_X+\Delta)=\kappa (Y, K_Y+\Delta_Y)\geq 
\kappa (Y, K_Y)\geq 0. 
$$ 
Therefore, from now on, 
we assume that $\kappa (Y, K_Y)=-\infty$. 
By Lemma \ref{f-lem4.3}, $Y$ is projective. 
Therefore, 
by Lemma \ref{f-lem4.1}, 
$X$ is projective since $X$ is $\mathbb Q$-factorial by 
assumption. 
Thus, by \cite[Theorem 5.1]{fujino-surfaces}, 
we get $\kappa(X, K_X+\Delta)\geq 0$. 
\end{proof}

The following theorem is the main result of this section, 
which is the abundance theorem for $\mathbb Q$-factorial 
log surfaces in Fujiki's class $\mathcal C$. 

\begin{thm}[Abundance theorem for 
$\mathbb Q$-factorial log surfaces in Fujiki's class 
$\mathcal C$]\label{h-thm7.2}
Let $(X, \Delta)$ be a $\mathbb Q$-factorial 
log surface in Fujiki's class $\mathcal C$. 
Assume that 
$(K_X+\Delta)\cdot C\geq 0$ for 
every curve $C$ on $X$. 
Then $K_X+\Delta$ is semi-ample. 
\end{thm}

For the proof of Theorem \ref{h-thm7.2}, 
we prepare some easy lemmas. 

\begin{lem}\label{h-lem7.3}
Let $X$ be a compact normal analytic surface and 
let $\mathcal L$ be a line bundle on $X$ such that 
$\mathcal L\cdot C\geq 0$ for 
every curve $C$ on $X$. 
Assume that $\kappa (X, \mathcal L)=1$. 
Then $\mathcal L$ is semi-ample. 
\end{lem}
\begin{proof}
This is an easy consequence of Zariski's lemma (see 
\cite[Chapter III, (8.2) Lemma]{bhpv}). 
For the details, see \cite[(4.1) Theorem]{fujita}. 
\end{proof}

\begin{lem}\label{h-lem7.4}
Let $S$ be a compact smooth analytic 
surface in Fujiki's class $\mathcal C$ 
with $\kappa (S, K_S)=0$. 
Assume that the algebraic dimension $a(S)$ of $S$ is less than 
two. 
Then $S$ is bimeromorphically equivalent to a $K3$ surface 
or a two-dimensional complex torus. 
\end{lem}
\begin{proof}
Since $S$ is in Fujiki's class $\mathcal C$, 
the first Betti number $b_1(S)$ of $S$ is 
even. 
Then the Enriques--Kodaira classification 
(see \cite[Chapter VI]{bhpv}) 
and $\kappa (S, K_S)=0$ give the 
desired statement. 
\end{proof}

\begin{lem}\label{h-lem7.5}
Let $B$ be a non-zero effective divisor on a two-dimensional 
complex torus $S$. Then we have $\kappa (S, B)\geq 1$. 
\end{lem}

\begin{proof}
Without loss of generality, we may assume that 
$B$ is an irreducible curve on $S$. 
If $B$ is not an elliptic curve, then we can 
see that $S$ is an Abelian surface 
(see \cite[Lemma 10.8]{ueno}). 
In this case, it is well known that 
$|2B|$ is basepoint-free. In particular, 
$\kappa (S, B)\geq 1$. 
Therefore, from now on, we assume that 
$B$ is an elliptic curve. 
By taking a suitable translation, we may further 
assume that $B$ is a complex subtorus of $S$. 
We set $A=S/B$. 
Let $p\colon S\to A$ be the canonical quotient map. 
Then $B=p^*P$ holds for $P=p(B)\in A$. 
Therefore, we obtain $\kappa (S, B)=\kappa (A, P)=1$. 
Hence, we always have $\kappa (S, B)\geq 1$. 
\end{proof}

\begin{lem}\label{h-lem7.6}
Let $S$ be a $K3$ surface and let 
$B$ be a non-zero effective divisor on $S$ such that $B^2=0$. 
Then we have $\kappa (S, B)\geq 1$. 
\end{lem}
\begin{proof}
By the Riemann--Roch formula, 
$$
\dim H^0(S, \mathcal O_S(B))+\dim H^2(S, \mathcal O_S(B))\geq 
\chi (S, \mathcal O_S)=2. 
$$
By Serre duality, 
$$
H^2(S, \mathcal O_S(B))\simeq 
H^0(S, \mathcal O_S(-B)). 
$$ 
Since $B$ is a non-zero effective divisor on $S$, 
$H^0(S, \mathcal O_S(-B))=0$ and $\dim H^0(S, \mathcal O_S(B))\geq 2$. 
Therefore, we have $\kappa (S, B)\geq 1$. 
\end{proof}

Before we prove Theorem \ref{h-thm7.2}, 
we explicitly state the abundance theorem for 
log canonical surfaces in Fujiki's class $\mathcal C$. 

\begin{cor}[Abundance theorem for 
log canonical surfaces in Fujiki's class $\mathcal C$]\label{h-cor7.7}
Let $(X, \Delta)$ be a log canonical surface in 
Fujiki's class $\mathcal C$. 
Assume that 
$(K_X+\Delta)\cdot C\geq 0$ for every 
curve $C$ on $X$. 
Then $K_X+\Delta$ is semi-ample. 
\end{cor}
\begin{proof}
Let $f\colon Y\to X$ be the minimal resolution of $X$. 
We put $K_Y+\Delta_Y:=f^*(K_X+\Delta)$. 
Then $\Delta_Y$ is effective by the negativity lemma 
and is a subboundary $\mathbb Q$-divisor 
on $Y$ since $(X, \Delta)$ is log canonical. 
Therefore, $\Delta_Y$ is a boundary $\mathbb Q$-divisor on $Y$. 
We can easily see that 
$(K_Y+\Delta_Y)\cdot C_Y\geq 0$ for every 
curve $C_Y$ on $Y$. 
Thus, by Theorem \ref{h-thm7.2}, 
$K_Y+\Delta_Y$ is semi-ample. 
This implies that $K_X+\Delta$ is also semi-ample. 
\end{proof}

Let us start the proof of Theorem \ref{h-thm7.2}. 

\begin{proof}[Proof of Theorem \ref{h-thm7.2}]
By the non-vanishing theorem (see Theorem \ref{h-thm7.1}), 
we have $\kappa (X, K_X+\Delta)\geq 0$. 
\setcounter{step}{0}
\begin{step}[$\kappa=2$]
If $\kappa(X, K_X+\Delta)=2$, 
then $X$ is projective by Corollary \ref{f-cor4.2}. 
In this case, we can apply \cite[Theorem 4.1]{fujino-surfaces}, 
which is one of the deepest results in \cite{fujino-surfaces}, 
and obtain that $K_X+\Delta$ is semi-ample. 
\end{step}
\begin{step}[$\kappa=1$]
If $\kappa(X, K_X+\Delta)=1$, then we see 
that $K_X+\Delta$ is semi-ample by 
Lemma \ref{h-lem7.3}. 
\end{step}
\begin{step}[$\kappa=0$]
In this step, we assume $\kappa(X, K_X+\Delta)=0$. 
If $X$ is projective, then $K_X+\Delta$ is semi-ample 
by \cite[Theorem 6.2]{fujino-surfaces}. 
Here, we will explain that the proof of \cite[Theorem 6.2]{fujino-surfaces} 
works with some minor modifications when $X$ is not projective. 
From now on, we 
will freely use the notation of the proof of \cite[Theorem 6.2]{fujino-surfaces}. 

The first part of the proof of \cite[Theorem 6.2]{fujino-surfaces} 
works without any changes (see page 361 in \cite{fujino-surfaces}). 
We note that $Z$ is a member of $|m(K_S+\Delta_S)|$ for some divisible 
positive integer $m$. We also note that Mumford's arguments 
on {\em{indecomposable curves of canonical type}} 
work on smooth analytic surfaces 
(see \cite[Definition, Lemma, and Corollary 1 in Section 2]{mumford}). 
Therefore, \cite[Lemma 6.3]{fujino-surfaces} 
holds true. In particular, we obtain that $Z^2=0$. 
The compact smooth surface $S$ 
constructed in the first part of \cite[Theorem 6.2]{fujino-surfaces} 
is not projective. Of course, $S$ is in Fujiki's class $\mathcal C$ because 
$S$ is bimeromorphically equivalent to $X$ by construction. 
As in the proof of \cite[Theorem 6.2]{fujino-surfaces}, 
we will derive a contradiction assuming $Z\ne 0$. 

By Lemma \ref{f-lem4.3}, 
we have $\kappa (S, K_S)\geq 0$ since 
$S$ is not projective. 
Thus, all we have to 
do is to check 
that Step 1 in the proof of \cite[Theorem 6.2]{fujino-surfaces} 
works when $S$ is not projective. 

In Step 1 in the proof of \cite[Theorem 6.2]{fujino-surfaces}, 
$S$ is a compact smooth analytic surface with $\kappa (S, K_S)=0$ and there 
are no $(-1)$-curves on $S$. 
Since $S$ is in Fujiki's class $\mathcal C$, 
the first Betti number $b_1(S)$ of $S$ is even. 
Therefore, by the Enriques--Kodaira classification, 
$S$ is a $K3$ surface or a complex torus (see Lemma \ref{h-lem7.4}). 
Then, by Lemmas \ref{h-lem7.5} and \ref{h-lem7.6}, 
we have $\kappa (X, K_X+\Delta)=
\kappa (S, K_S+\Delta_S)=\kappa (S, Z)\geq 1$ and 
get a contradiction. This means that 
Step 1 in the proof of \cite[Theorem 6.2]{fujino-surfaces} works 
when $S$ is not projective. 
\end{step}

Therefore, $K_X+\Delta$ is always semi-ample. 
This is what we wanted. 
\end{proof}

\section{Contraction theorem 
for log canonical surfaces}\label{f-sec8} 

In this section, we discuss a contraction theorem 
for log canonical surfaces. 
Note that compact log canonical surfaces are not necessarily 
$\mathbb Q$-factorial. 
Therefore, we need Mumford's intersection theory 
(see \cite{mumford1}, \cite[Remark 4-6-3]{matsuki}, 
and \cite{sakai}). 

\begin{defn}[Mumford's intersection theory]\label{f-def8.1} 
Let $X$ be a normal analytic surface and let 
$\pi\colon Y\to X$ be a resolution. 
Let $\Exc(\pi)=\sum_i E_i$ be the irreducible decomposition 
of the exceptional curve of $\pi$. 
Let $D$ be a $\mathbb Q$-divisor 
on $X$. 
Then we can define the {\em{inverse image}} 
$\pi^*D$ 
as 
$$
\pi^*D=D^\dag+\sum _i \alpha_i E_i, 
$$
where $D^\dag$ is the strict transform of $D$ by 
$\pi$ and the rational numbers $\alpha_i$ are uniquely 
determined by 
the following linear equations: 
$$
D^\dag\cdot E_j+\sum _i \alpha_i E_i \cdot E_j=0
$$
for every $j$. 
We call $\pi^*D$ the {\em{pull-back of 
$D$ in the sense of Mumford}}. 
Of course, $\pi^*D$ coincides with 
the usual one when $D$ is $\mathbb Q$-Cartier. 

From now on, we further assume that $X$ is compact. 
The {\em{intersection number}} $D\cdot D'$ {\em{$($in the sense of 
Mumford$)$}} is 
defined to be the rational number $(\pi^*D)\cdot (\pi^*D')$, 
where $D$ and $D'$ are $\mathbb Q$-divisors on $X$. 
We can easily see that $D\cdot D'$ is well-defined. 
We note that it coincides with the usual one when $D$ or $D'$ 
is $\mathbb Q$-Cartier. 
\end{defn}

Let us recall some definitions and basic properties 
of surface singularities for the reader's convenience. 

\begin{defn}[{Numerically log canonical and numerically dlt, 
see \cite[Notation 4.1]{kollar-mori}}]\label{f-def8.2}
Let $X$ be a normal analytic surface and let $\Delta$ be a 
$\mathbb Q$-divisor on $X$. Let $f\colon Y\to U\subset X$ be a proper 
bimeromorphic morphism from a smooth 
surface $Y$ to an open set $U$ of $X$. 
Then we can define $f^*(K_U+\Delta|_U)$ in the sense of 
Mumford (see Definition \ref{f-def8.1}) without assuming that 
$K_U+\Delta|_U$ is $\mathbb Q$-Cartier. 
Thus we can always 
write 
$$
K_Y=f^*(K_U+\Delta|_U)+\sum _{E_i} a(E_i, X, \Delta)E_i
$$
such that $f_*\left(\sum _{E_i} a(E_i, X, \Delta)E_i\right)=-\Delta|_U$. 
If $\Delta$ is effective and $a(E_i, X, \Delta)\geq -1$ 
for every exceptional curve $E_i$ and $f\colon Y\to U\subset X$, 
then we say that $(X, \Delta)$ is {\em{numerically 
log canonical}}. 
We say that $(X, \Delta)$ is {\em{numerically dlt}} 
if $(X, \Delta)$ is numerically log canonical and 
there exists a finite set $Z\subset X$ such that 
$X\setminus Z$ is smooth, 
$\Supp \Delta|_{X\setminus Z}$ is a simple 
normal crossing divisor on $X\setminus Z$, 
and $a(E, X, \Delta)>-1$ for every 
exceptional curve $E$ which maps to $Z$. 
It is well known that if $(X, \Delta)$ is numerically 
log canonical then $K_X+\Delta$ is $\mathbb Q$-Cartier (see 
\cite[Proposition 3.5]{fujino-surfaces} and \cite[Remark 4-6-3]{matsuki}). 
Moreover, if $(X, \Delta)$ is numerically dlt then 
$X$ has only rational singularities (see \cite[Theorem 4.12]{kollar-mori}).  
\end{defn}

\begin{rem}\label{f-rem8.3}
In Definition \ref{f-def8.2}, we only require that 
$\Supp \Delta|_{X\setminus Z}$ is 
a simple normal crossing divisor on $X\setminus Z$ in the classical topology. 
So it permits some irreducible component of 
$\Supp \Delta|_{X\setminus Z}$ to have nodal singularities. 
Therefore, our definition does not coincide with 
\cite[Notation 4.1]{kollar-mori} when $X$ is an algebraic surface. 
However, since we are mainly interested in local 
analytic properties of singularities of pairs $(X, \Delta)$, 
this difference causes no subtle problems. 
\end{rem}

We need the following 
contraction theorem for log canonical surfaces 
in Sections \ref{f-sec9} and 
\ref{f-sec10}. 

\begin{thm}[{Contraction theorem for log canonical surfaces, 
see \cite[Theorem 4.1]{fujino-rational}}]\label{f-thm8.4}
Let $(X, \Delta)$ be a compact log canonical 
surface and let $C$ be an irreducible 
curve on $X$ such that 
$-(K_X+\Delta)\cdot C>0$ and 
$C^2<0$, where $C^2$ is the self-intersection number of 
$C$ in the sense of Mumford $($see Definition \ref{f-def8.1}$)$. 
Then we have a projective bimeromorphic morphism 
$\varphi\colon X\to Y$ onto a normal surface $Y$ such that 
$\Exc(\varphi)=C\simeq \mathbb P^1$ and that 
$C$ passes through no non-rational singular points of $X$, that is, 
$X$ has only rational singularities in a neighborhood 
of $C$. In particular, $C$ is $\mathbb Q$-Cartier. Moreover, 
$(Y, \Delta_Y)$ is log canonical with $\Delta_Y:=\varphi_*\Delta$. 
\end{thm}

\begin{proof}
By Sakai's contraction theorem (see \cite[Theorem (1.2)]{sakai}), 
we have a bimeromorphic morphism 
$\varphi\colon X\to Y$ which contracts $C$ to a normal point $P\in Y$. 
Since $-(K_X+\Delta)\cdot C>0$, $(Y, \Delta_Y)$ is 
numerically dlt in a neighborhood 
of $P$ by the negativity lemma. 
Therefore, $K_Y+\Delta_Y$ is $\mathbb Q$-Cartier and 
$Y$ has only rational singularities in a neighborhood 
of $P$. Of course, $(Y, \Delta_Y)$ is a compact log canonical surface. 
By Theorem \ref{f-thm11.3} below, 
$R^i\varphi_*\mathcal O_X=0$ for 
every $i>0$. 
Thus, $X$ has only rational singularities in a neighborhood 
of $C$ by Lemma \ref{h-lem2.8}. In particular, 
$C$ is $\mathbb Q$-Cartier (see \cite[Chapter II, 2.12.~Lemma]
{nakayama}). 
Since $R^1\varphi_*\mathcal O_X=0$, we can easily check that 
$C\simeq \mathbb P^1$ as in 
Claim \ref{g-claim1} of the proof of Theorem \ref{f-thm5.1}. 
We see that 
$\varphi$ is projective, $-(K_X+\Delta)$ and $-C$ 
are $\varphi$-ample by construction. 
\end{proof}

We close this section with simple but very important remarks. 

\begin{rem}[Extremal rays]\label{f-rem8.5} 
Theorem \ref{f-thm8.4} says that 
$X$ has only rational singularities in a neighborhood 
of the exceptional curve $C$ and then 
$C$ is automatically $\mathbb Q$-Cartier. 
Therefore, if $X$ is projective, 
then $C$ spans a $(K_X+\Delta)$-negative 
extremal ray $R$ of $\NE(X)$ 
in the usual sense (see Lemma \ref{f-lem5.2}). 
Thus the contraction $\varphi$ in Theorem \ref{f-thm8.4} 
is nothing but the usual contraction morphism 
associated to the extremal ray $R$. 
In particular, $Y$ is also projective when so is $X$. 
\end{rem}

\begin{rem}[Termination of contractions]\label{f-rem8.6} 
Assume that $X$ is Moishezon. 
We consider a sequence of contraction morphisms 
as in Theorem \ref{f-thm8.4} 
$$
(X, \Delta)=:(X_0, \Delta_0)\overset{\varphi_0}{\longrightarrow} 
(X_1, \Delta_1) \overset{\varphi_1}{\longrightarrow} 
\cdots 
\overset{\varphi_{i-1}}{\longrightarrow} 
(X_i, \Delta_i)\overset{\varphi_{i}}{\longrightarrow} \cdots
$$ 
starting from a log canonical surface $(X_0, \Delta_0):=(X, \Delta)$. 
Let $C_i$ be the $\varphi_i$-exceptional curve for every $i$. 
By Theorem \ref{f-thm8.4}, 
$C_i$ is $\mathbb Q$-Cartier for every $i$. 
Then we can easily see that 
$C_0$, $\varphi_0^*C_1$, $\ldots$, $\varphi_0^*\cdots\varphi_{i-1}^*C_i$, 
$\ldots$ are linearly independent in $N^1(X)$. 
Therefore, the sequence must terminate since $\rho(X)<\infty$. 
\end{rem}

\section{Log canonical surfaces in Fujiki's class $\mathcal C$ with negative 
Kodaira dimension}\label{f-sec9} 

The main purpose of this section is to prove the following 
theorem. 

\begin{thm}[see Theorem \ref{f-thm1.3}]\label{f-thm9.1}
Let $(X, \Delta)$ be a log canonical surface in Fujiki's class $\mathcal C$. 
Assume that $\kappa (X, K_X+\Delta)=-\infty$ holds. 
Then $X$ is projective. 
\end{thm}

Let us recall the following well-known lemma for the reader's convenience 
(see \cite[Remark 4-6-29]{matsuki}). 

\begin{lem}\label{f-lem9.2}
Let $(X, \Delta)$ be a log canonical surface. 
Assume that $P\in X$ is not a rational singularity. 
Then $P\not\in \Supp \Delta$ and $X$ is Gorenstein at $P$. 
\end{lem}
\begin{proof}[Sketch of Proof]
If $P\in \Supp \Delta$, 
then $(X, 0)$ is numerically dlt in a neighborhood of $P$. 
In particular, $X$ has only rational singularities in a neighborhood of 
$P$. Therefore, we have $P\not\in \Supp \Delta$. 
By the classification of two-dimensional log canonical 
singularities (see \cite[Theorem 4.7]{kollar-mori} 
and \cite[Theorem 4-6-28]{matsuki}), 
$P\in X$ is a simple elliptic singularity or a cusp singularity 
(see \cite[Note 4.8]{kollar-mori} and \cite[Theorem 4-6-28]{matsuki}). 
We can check that all the other two-dimensional 
log canonical singularities are rational singularities 
(see \cite[Remark 4-6-29]{matsuki}). 
Therefore, $X$ is Gorenstein at $P$. 
\end{proof}

Let us start the proof of Theorem \ref{f-thm9.1}. 

\begin{proof}[Proof of Theorem \ref{f-thm9.1}]
We divide the proof into several small steps. 
\setcounter{step}{0}
\begin{step}\label{f-thm9.1-step1}
In this step, we will prove that $X$ is Moishezon, 
that is, the algebraic dimension 
$a(X)$ of $X$ is two. 

Let $f\colon Y\to X$ be the minimal resolution of $X$ with 
$K_Y+\Delta_Y:=f^*(K_X+\Delta)$. 
Then $(Y, \Delta_Y)$ is log canonical since 
so is $(X, \Delta)$ by assumption. 
By applying Lemma \ref{f-lem4.3} to 
$(Y, \Delta_Y)$, we obtain that 
$Y$ is a smooth projective 
surface. 
This implies that 
$X$ is Moishezon, that is, the algebraic dimension 
$a(X)$ of $X$ is two. 
\end{step}
\begin{step}\label{f-thm9.1-step2}
If $X$ has only rational singularities, then $X$ is 
$\mathbb Q$-factorial (see Lemma \ref{f-lem3.10}). 
In this case, by Lemma \ref{f-lem4.1}, 
$X$ is projective since we have already 
known that $X$ is Moishezon in Step \ref{f-thm9.1-step1}. 
\end{step}
Therefore, from now on, we may assume that $X$ has 
at least one non-rational singular point. 
\begin{step}\label{f-thm9.1-step3}
By applying Theorem \ref{f-thm8.4} finitely many times, 
we may assume that if $C$ is an irreducible curve on $X$ 
with $-(K_X+\Delta)\cdot C>0$ then 
$C^2\geq 0$ holds (see Remark \ref{f-rem8.6}). 
\end{step}
\begin{step}\label{f-thm9.1-step4}
Let $g\colon Z\to X$ be the minimal resolution of non-rational 
singularities of $X$. 
Then we get the following commutative diagram. 
$$
\xymatrix{
Y \ar[d]_h\ar[dr]^f& \\ 
Z \ar[r]_g& X
}
$$
Since $Z$ has only rational singularities by 
construction, we see that $Z$ is $\mathbb Q$-factorial 
(see Lemma \ref{f-lem3.10}). 
Therefore, $Z$ 
is projective by Lemma \ref{f-lem4.1} since 
$Z$ is Moishezon (see Step \ref{f-thm9.1-step1}). 
We put $K_Z+\Delta_Z:=g^*(K_X+\Delta)$. 
Then $\Delta_Z$ is effective by the negativity lemma. 
Of course, $(Z, \Delta_Z)$ is log canonical. 
More precisely, by Lemma \ref{f-lem9.2}, 
we have $\Delta_Z=\sum _i E_i +g_*^{-1}\Delta$, 
where $\Exc(g)=\sum _i E_i$. 
Since $K_Z+\Delta_Z$ never becomes nef by 
$\kappa(Z, K_Z+\Delta_Z)=\kappa(X, K_X+\Delta)=-\infty$ 
(see Theorem \ref{h-thm7.1}), 
there exists a $(K_Z+\Delta_Z)$-negative 
extremal ray $R=\mathbb R_{\geq 0}[\bar{C}]$ of 
$\NE(Z)$, where $\bar{C}$ is an irreducible rational curve 
on $Z$. 
\end{step}
\begin{step}\label{f-thm9.1-step5}
In this step, we will prove the following claim. 
\begin{claim}\label{g-claim2} 
The self-intersection number $\bar{C}^2$ of 
$\bar{C}$ is non-negative, where $\bar{C}$ is 
a $(K_Z+\Delta_Z)$-negative 
extremal rational curve as in Step \ref{f-thm9.1-step4}. 
\end{claim}
\begin{proof}[Proof of Claim \ref{g-claim2}]
We assume that $\bar{C}^2<0$ holds. 
Since $-(K_Z+\Delta_Z)\cdot \bar{C}>0$, $\bar{C}$ is not 
$g$-exceptional. 
If $C_X:=g_*\bar{C}$ is disjoint from non-rational singularities 
of $X$, then we have $(C_X)^2<0$ and 
$-(K_X+\Delta)\cdot C_X>0$ because $g$ is an isomorphism 
in a neighborhood of $\bar{C}$. 
This is a contradiction by Step \ref{f-thm9.1-step3}. 
Therefore, $C_X$ passes though 
at least one non-rational singular point $P$ of $X$. 
This implies that $C_X\not \subset \Supp \Delta$ by 
Lemma \ref{f-lem9.2}.  
Let $C_Y$ be 
the strict transform of $\bar{C}$ on $Y$. 
Then we can easily see that $(C_Y)^2<0$ holds and 
that $C_Y$ is not contained in $\Supp \Delta_Y$. 
Thus we obtain that $-K_Y
\cdot C_Y\geq -(K_Y+\Delta_Y)\cdot C_Y=-(K_Z+\Delta_Z)
\cdot \bar{C}>0$. 
This means that $C_Y$ is a $(-1)$-curve on $Y$. 
In particular, $-K_Y\cdot C_Y=1$. 
On the other hand, $\Delta_Y\cdot C_Y\geq 1$ since $C_X$ 
passes through a non-rational singular point $P$ and 
the reduced part of $\Delta_Y$ contains $f^{-1}(P)$. 
Thus we have $-(K_Y+\Delta_Y)\cdot C_Y=1-\Delta_Y\cdot C_Y\leq 0$. 
This is a contradiction. Therefore, $\bar{C}^2\geq 0$ holds. 
\end{proof}
Therefore, every $(K_Z+\Delta_Z)$-negative extremal ray  
is spanned by an irreducible rational curve $\bar{C}$ with $\bar{C}^2\geq 0$.  
\end{step}

We will treat the case where $\bar{C}^2>0$ and $\bar{C}^2=0$ 
in Step \ref{f-thm9.1-step6} and Step \ref{f-thm9.1-step7}, respectively. 

\begin{step}\label{f-thm9.1-step6} 
We assume that 
there exists a $(K_Z+\Delta_Z)$-negative 
extremal ray $R$ of $\NE(Z)$ spanned by 
an irreducible rational curve $\bar{C}$ with $\bar{C}^2>0$. 
In this case, 
the numerical equivalence class of $\bar{C}$ is an interior 
point of $\NE(Z)$ (see \cite[Corollary 1.21]{kollar-mori}) and 
it also generates an extremal ray of $\NE(Z)$. 
Therefore, it is easy to see that $\rho(Z)=1$ and 
$-(K_Z+\Delta_Z)$ is 
ample. This is a contradiction because $-(K_Z+\Delta_Z)\cdot E=0$ 
for every $g$-exceptional curve $E$ on $Z$. 
Thus, this case does not happen. 
\end{step}
\begin{step}\label{f-thm9.1-step7}
Hence we may assume that there exists a $(K_Z+\Delta_Z)$-negative 
extremal ray $R$ spanned by an irreducible rational curve 
$\bar{C}$ with $\bar{C}^2=0$. 
Then there exists a surjective morphism $\varphi_R\colon Z\to W$ onto a 
smooth projective curve $W$ 
such that $\rho(Z/W)=1$ and $-(K_Z+\Delta_Z)$ is 
$\varphi_R$-ample, that is, 
$\varphi_R\colon Z\to W$ is a Mori fiber space. 
Without loss of generality, we may assume that $\bar{C}$ is 
a general fiber of $\varphi_R$. 
Since the self-intersection number of any irreducible 
component of $\Exc(g)$ is negative, 
every irreducible component of $\Exc(g)$ is dominant onto 
$W$ by $\varphi_R$. 
Since 
$-K_Z\cdot \bar{C}=2$ and 
$\Delta_Z\geq \sum _i E_i=\Exc(g)$, 
we can check that $E:=\Exc(g)$ is an irreducible 
curve such that $\varphi_R|_E:E\to W$ is an isomorphism. 
By the classification of two-dimensional log canonical 
singularities (see \cite[Theorem 4.7]{kollar-mori} 
and \cite[Theorem 4-6-28]{matsuki}), 
$E$ is an elliptic curve and $P$ is a simple elliptic 
singularity. 
\begin{claim}\label{g-claim3}
$-K_X$ is ample, that is, $X$ is Fano. In particular, $X$ is projective. 
\end{claim}
\begin{proof}[Proof of Claim \ref{g-claim3}]
Since $\varphi_R\colon Z\to W$ is a Mori fiber space 
and $-(K_Z+E)\cdot \bar{C}=E\cdot \bar{C}=1$, 
we obtain that 
$-(K_Z+E)$ is $\mathbb Q$-linearly equivalent to 
$E+\varphi_R^*D$, where 
$D$ is some $\mathbb Q$-divisor on $W$. 
We note that $\deg D=-E^2>0$ since $(E+\varphi_R^*D)\cdot 
E=-(K_Z+E)\cdot E=0$. 
Thus, we have $\kappa (Z, -(K_Z+E))\geq 0$. 
Since 
$K_Z+E=g^*K_X$, we obtain  
$\kappa (X, -K_X)=\kappa (Z, -(K_Z+E))\geq 0$. 
We take any irreducible curve $C$ on $X$. 
Let $C'$ be the strict transform of $C$ on $Z$. 
Then $-K_X\cdot C=-(K_Z+E)\cdot C'=(E+\varphi_R^*D)\cdot 
C'>0$ since $C'\ne E$. 
Thus, we obtain that $-K_X$ is 
ample by Nakai--Moishezon's ampleness 
criterion for algebraic spaces 
(see \cite[(1.4) Theorem]{pascul}). 
\end{proof}
Thus we see that $X$ is projective when $X$ has at least one 
non-rational singular point. 
\end{step}
Therefore, we obtain that $X$ is always projective if 
$(X, \Delta)$ is a log canonical surface in Fujiki's 
class $\mathcal C$ with $\kappa (X, K_X+\Delta)=-\infty$. 
\end{proof}

\begin{rem}\label{f-rem9.3}
Let $(X, \Delta)$ be a log canonical surface in Fujiki's class 
$\mathcal C$ with $\kappa (X, K_X+\Delta)=-\infty$. 
The proof of Theorem \ref{f-thm9.1} 
shows that if $X$ has non-rational singularities 
then $X$ has one simple elliptic singularity 
and no cusp singularities. 
\end{rem}

By Theorem \ref{f-thm9.1}, we can freely apply the minimal model 
theory of projective log canonical surfaces established in 
\cite{fujino-surfaces} 
to log canonical surfaces in Fujiki's class $\mathcal C$ with 
negative Kodaira dimension. 

\begin{rem}\label{f-rem9.4}
We can construct a complete 
non-projective log canonical algebraic surface 
$(X, \Delta)$ with $\kappa (X, K_X+\Delta)\geq 0$. 
For some examples, see Section \ref{p-sec12} below. 
Therefore, the assumption $\kappa (X, K_X+\Delta)=-\infty$ 
is indispensable in Theorem \ref{f-thm9.1}. 
\end{rem}

\section{Proof of Theorem \ref{f-thm1.5}}\label{f-sec10}

In this section, we prove Theorem \ref{f-thm1.5}, 
that is, the minimal model theory for log canonical 
surfaces in Fujiki's class $\mathcal C$. 
We give a detailed proof for the reader's convenience, 
although it is essentially the same as that of 
Theorem \ref{f-thm1.1}. 

\begin{proof}[Proof of Theorem \ref{f-thm1.5}]
If $(K_X+\Delta)\cdot C\geq 0$ for every curve $C$ on $X$, 
then $K_X+\Delta$ is semi-ample by Corollary \ref{h-cor7.7}. 
So $(X, \Delta)$ is itself a good minimal model 
of $(X, \Delta)$. 
Therefore, we may assume that $(K_X+\Delta)\cdot C<0$ 
for some curve $C$ on $X$. 
If $X$ is projective, 
then we can run the minimal model program for projective 
log canonical surfaces and finally get a good minimal model 
or a Mori fiber space (see \cite{fujino-surfaces}). 
Thus we may assume that 
$X$ is not projective. 
By Theorem \ref{f-thm9.1}, we obtain $\kappa (X, K_X+\Delta)\geq 0$. 
Therefore, we have an effective Cartier divisor 
$D\in |m(K_X+\Delta)|$ for some positive divisible 
integer $m$. 
Since $(K_X+\Delta)\cdot C<0$ for some curve $C$ on $X$, 
$C$ is an irreducible component of $\Supp D$ such that 
the self-intersection number $C^2$ is negative. 
We apply the contraction theorem:~Theorem \ref{f-thm8.4}. 
Since there are only finitely many irreducible components of 
$\Supp D$, we finally get $(X^*, \Delta^*)$ such that 
$K_{X^*}+\Delta^*$ is semi-ample after finitely many contractions. 

Since $R^1{\varphi_i}_*\mathcal O_{X_i}=0$ by Theorem \ref{f-thm11.3} 
below, 
$X_i$ has only rational singularities if and only if so does 
$X_{i+1}$ by Lemma \ref{h-lem2.8}. 
Therefore, we have (2). 

Since $\varphi_i$ is projective by construction, 
$X_i$ is projective when so is $X_{i+1}$. 
If $X_i$ is projective, then $\varphi_i$ is the usual 
contraction morphism associated to a $(K_{X_i}+\Delta_i)$-negative 
extremal ray (see Remark \ref{f-rem8.5}). 
This implies that $X_{i+1}$ is also projective. 
Thus we have (1). 

Since $\Exc(\varphi_i)\simeq \mathbb P^1$, $C_i=\Exc(\varphi_i)$ 
is $\mathbb Q$-Cartier, 
and $-C_i$ is $\varphi_i$-ample, 
we can easily check (3) by Theorem \ref{f-thm3.11}. 
\end{proof}

\section{Appendix:~Vanishing theorems}\label{f-sec11} 

In this section, we quickly explain 
some vanishing theorems for the reader's convenience. 
Fortunately, we do not need 
difficult analytic methods. They follow from elementary arguments. 

\medskip 

Let us start with the following vanishing theorem. 
We learned it from \cite{kollar-kovacs} (see \cite[Theorem 10.4]{kollar2}). 

\begin{thm}[Relative vanishing theorem]\label{f-thm11.1} 
Let $\varphi\colon V\to W$ be a proper bimeromorphic 
morphism from a smooth analytic surface to a normal 
analytic surface $W$. Assume 
that there exists a point $P\in W$ such that 
$\varphi$ is an isomorphism 
over $W\setminus P$. 
Let $\Exc(\varphi) =\sum _i E_i$ be the irreducible decomposition of 
the $\varphi$-exceptional locus $\Exc(\varphi)$. 
Let $\mathcal L$ be a line bundle on $V$, 
let $N$ be a $\mathbb Q$-divisor on $V$, and 
let $E=\sum _i b_i E_i$ be an effective $\mathbb Q$-divisor 
on $V$. 
Assume that $N\cdot E_i\geq 0$ and $\mathcal L\cdot E_i=
(K_V+E+N)\cdot E_i$ hold for 
every $i$. 
We further assume that one of the following 
conditions holds. 
\begin{itemize}
\item[(1)] $0\leq b_i<1$ for every $i$. 
\item[(2)] $0<b_i\leq 1$ for every $i$ and 
there exists some $j$ such that 
$b_j\ne 1$. 
\item[(3)] $0<b_i\leq 1$ for every $i$ and 
there exists some $j$ such that $N\cdot E_j>0$. 
\end{itemize}
Then $R^i\varphi_*\mathcal L=0$ holds for every $i>0$. 
\end{thm}

Theorem \ref{f-thm11.1} (1) is a special case 
of the Kawamata--Viehweg vanishing theorem. 
Note that condition (1) is most useful. 

\begin{rem}\label{f-rem11.2}
In Theorem \ref{f-thm11.1}, 
it is sufficient to assume that $N$ is a $\mathbb Q$-line bundle 
on $V$, that is, 
a finite $\mathbb Q$-linear combination 
of some line bundles on $V$ such that 
$N\cdot E_i\geq 0$ and $\mathcal L\cdot E_i=(K_V+E)\cdot 
E_i +N\cdot E_i$ hold for every $i$. 
\end{rem}

\begin{proof}[Proof of Theorem \ref{f-thm11.1}]
Let us consider $Z=\sum _i r_i E_i$, where 
$r_i$ is a non-negative integer for every $i$. 
Although \cite[Theorem 10.4]{kollar2} is formulated 
and proved for two-dimensional 
regular schemes, the proof works for two-dimensional 
complex manifolds. 
Therefore, by the proof of \cite[Theorem 10.4]{kollar2}, 
we obtain $H^1(Z, \mathcal L\otimes \mathcal O_Z)=0$ 
(see also \cite[2.2.1 Theorem]{kollar-kovacs}). 
Hence, by using the theorem on formal functions for 
proper morphisms between analytic spaces 
(see \cite[Chapter VI, Corollary 4.7]{banica-s}), 
we get $R^i\varphi_*\mathcal L=0$ for every $i>0$ 
(see also \cite[2.2.5 Corollary]{kollar-kovacs}). 
\end{proof}

As an application of Theorem \ref{f-thm11.1}, we 
can prove the following vanishing 
theorem, which is a Kawamata--Viehweg vanishing 
theorem for proper bimeromorphic 
morphisms between surfaces. 

\begin{thm}\label{f-thm11.3}
Let $X$ be a normal analytic surface and 
let $\Delta$ be an effective $\mathbb Q$-divisor on $X$ such that 
$K_X+\Delta$ is $\mathbb Q$-Cartier. 
Let $f\colon X\to Y$ be 
a proper bimeromorphic morphism onto a normal 
analytic surface $Y$. 
Let $\mathcal L$ be a line bundle on $X$ and let $D$ 
be a $\mathbb Q$-Cartier Weil divisor on $X$. 
Assume that one of the following conditions holds. 
\begin{itemize}
\item[(1)] $\mathcal L\cdot C+(D-(K_X+\Delta))
\cdot C>0$ for every $f$-exceptional 
curve $C$ on $X$ and the coefficients of $\Delta$ are 
less than or equal to one. 
\item[(2)] $\mathcal L\cdot C+(D-(K_X+\Delta))\cdot C\geq 0$ for every 
$f$-exceptional curve $C$ on $X$ and the 
coefficients of $\Delta$ are less than one. 
\end{itemize}
Then $R^if_*(\mathcal L\otimes \mathcal O_X(D))=0$ holds 
for every $i>0$. 
\end{thm}

Let us prove Theorem \ref{f-thm11.3}. 
The following proof is essentially the same as that of 
\cite[Theorem 6.2]{fujino-tanaka}. 

\begin{proof}[Proof of Theorem \ref{f-thm11.3}]
We divide the proof into small steps. 
\setcounter{step}{0} 
\begin{step}
Without loss of generality, we can freely shrink $Y$ and assume that 
$Y$ is a small relatively compact Stein open subset 
of normal analytic surface. 
We may further assume that 
$f$ is an isomorphism 
outside $P\in Y$ and $f^{-1}(P)$ 
is one-dimensional. 
\end{step}
\begin{step}
When $\lfloor \Delta\rfloor \ne 0$, 
we can take an $f$-ample 
Cartier divisor $H$ (see (1)). 
Then we can find an effective $\mathbb Q$-divisor $\Delta'$ 
on $X$ such that 
$\lfloor \Delta'\rfloor=0$ and that $\Delta'$ is $\mathbb Q$-linearly 
equivalent to 
$\Delta+\varepsilon H$ for some $0<\varepsilon \ll 1$. 
More precisely, we take a general member $B$ of 
$|\lfloor \Delta\rfloor +mH|$ for some large 
positive integer $m$ and put 
$$\Delta'=\Delta-\frac{1}{m}\lfloor \Delta\rfloor +\frac{1}{m}B. 
$$ 
By replacing $\Delta$ with $\Delta'$, 
we can always assume that $\lfloor \Delta\rfloor=0$. 
\end{step}
\begin{step}
Let $\varphi\colon Z\to X$ be the minimal resolution of $X$. 
We set $K_Z+\Delta_Z:=\varphi^*(K_X+\Delta)$. 
We note that $\Delta_Z$ is effective by the negativity lemma. 
We note that 
$$
\varphi^*\mathcal L +\lceil \varphi^*D\rceil -(K_Z+\Delta_Z+\{-\varphi^*D\}) 
=\varphi^*(\mathcal L+D-(K_X+\Delta)). 
$$
We put $\Theta:=\Delta_Z+\{-\varphi^*D\}$. 
Then 
$$
\varphi^*\mathcal L +\lceil \varphi^*D\rceil -\lfloor \Theta
\rfloor=(K_Z+\{\Theta\})+\varphi^*(\mathcal L+D-(K_X+\Delta)). 
$$
We note that we can write $\{\Theta\}=E+M$ where 
$E$ is a $\varphi$-exceptional 
effective $\mathbb Q$-divisor 
with $\lfloor E\rfloor=0$ and 
$M$ is an effective $\mathbb Q$-divisor such that 
every irreducible component of $M$ is not $\varphi$-exceptional. 
Let us consider 
$$
0\to \varphi^*\mathcal L\otimes \mathcal O_Z(\lceil \varphi^*D\rceil 
-\lfloor \Theta\rfloor)\to 
\varphi^*\mathcal L\otimes \mathcal O_Z(\lceil \varphi^*D\rceil )\to 
\varphi^*\mathcal L\otimes \mathcal O_{\lfloor \Theta\rfloor} 
(\lceil \varphi^*D\rceil)\to 0. 
$$
By Theorem \ref{f-thm11.1}, 
we have 
$$
R^1\varphi_*\left ( 
\varphi^*\mathcal L\otimes \mathcal O_Z(\lceil \varphi^*D\rceil 
-\lfloor \Theta\rfloor)
\right) =0. 
$$ 
Therefore, we get the following short exact sequence 
$$
0\to \mathcal L \otimes 
\varphi_* \mathcal O_Z(\lceil \varphi^*D\rceil -\lfloor \Theta\rfloor) 
\to 
\mathcal L \otimes 
\varphi_* \mathcal O_Z(\lceil \varphi^*D\rceil)  
\to 
\mathcal L \otimes 
\varphi_* \mathcal O_{\lfloor \Theta\rfloor} 
(\lceil \varphi^*D\rceil )\to 0.  
$$ 
By construction, $\lfloor \Theta\rfloor$ is $\varphi$-exceptional. 
Therefore, 
$\mathcal L \otimes 
\varphi_* \mathcal O_{\lfloor \Theta\rfloor} 
(\lceil \varphi^*D\rceil )$ is a skyscraper sheaf on $X$. 
Thus we obtain the following surjection 
\begin{equation}\label{eq-2} 
R^1f_*\left(\mathcal L \otimes 
\varphi_* \mathcal O_Z
(\lceil \varphi^*D\rceil -\lfloor \Theta \rfloor)\right) 
\to R^1f_*(\mathcal L\otimes \mathcal O_X(D))\to 0 
\end{equation}
since $\varphi_*\mathcal O_Z(\lceil \varphi^*D\rceil)\simeq 
\mathcal O_X(D)$. 
By the Leray spectral sequence, 
we have 
\begin{equation}\label{eq-3} 
R^1f_*\left(\mathcal L \otimes 
\varphi_* \mathcal O_Z
(\lceil \varphi^*D\rceil -\lfloor \Theta \rfloor)\right) 
\subset 
R^1(f\circ \varphi)_* 
\left( 
\varphi^*\mathcal L\otimes \mathcal O_Z(\lceil \varphi^*D\rceil 
-\lfloor \Theta\rfloor)\right). 
\end{equation}
As before, 
we can write $\{\Theta\}=E'+M'$ where 
$E'$ is a $f\circ \varphi$-exceptional 
effective $\mathbb Q$-divisor 
with $\lfloor E'\rfloor=0$ and 
$M'$ is an effective $\mathbb Q$-divisor such that 
every irreducible component of $M'$ is not $f\circ \varphi$-exceptional. 
By Theorem \ref{f-thm11.1}, 
we know that 
$$
R^1(f\circ \varphi)_* 
\left( 
\varphi^*\mathcal L\otimes \mathcal O_Z(\lceil \varphi^*D\rceil 
-\lfloor \Theta\rfloor)\right)=0. 
$$
This implies that 
$$
R^1f_*\left(\mathcal L \otimes 
\varphi_* \mathcal O_Z
(\lceil \varphi^*D\rceil -\lfloor \Theta \rfloor)\right)=0
$$ 
by \eqref{eq-3}. 
By the surjection \eqref{eq-2}, 
we get $R^1f_*(\mathcal L\otimes \mathcal O_X(D))=0$. 
\end{step}
\begin{step}
Since $f^{-1}(P)$ is one-dimensional, 
$R^if_*(\mathcal L\otimes \mathcal O_X(D))=0$ for 
every $i>2$ by 
the theorem on formal functions 
for proper morphisms between 
analytic spaces 
(see \cite[Chapter VI, Corollary 
4.7]{banica-s}). 
\end{step} 
Therefore, we have 
$R^if_*(\mathcal L\otimes \mathcal O_X(D))=0$ 
for every $i>0$. 
\end{proof}

We close this section with an obvious remark. 

\begin{rem}\label{f-rem11.4}
Theorems \ref{f-thm11.1} and \ref{f-thm11.3} can be formulated and 
proved easily for proper birational morphisms between algebraic surfaces 
defined over any algebraically closed field. 
Therefore, we can formulate and prove Theorem \ref{f-thm3.11} for 
projective birational morphisms 
between (not necessarily complete) algebraic surfaces defined over any algebraically 
closed field. This is because the proof of Theorem \ref{f-thm3.11} 
only needs the vanishing theorem:~Theorem \ref{f-thm11.3}. 
\end{rem}

\section{Appendix:~Complete non-projective algebraic surfaces}\label{p-sec12}

In this section, we construct some examples of complete non-projective 
log canonical algebraic surfaces. 
From Example \ref{p-ex12.1} to Example \ref{p-ex12.5}, 
we will work over $\mathbb C$, the complex number field. 
\medskip 

Let us start with Koll\'ar's example. Although it is not 
stated explicitly in \cite{kollar}, 
it does not satisfy Kleiman's ampleness criterion.  
We note that the arguments in Example \ref{p-ex12.3} below 
help the reader understand Example \ref{p-ex12.1}. 
Therefore, we do not explain the details of Koll\'ar's example. 

\begin{ex}[{\cite[Aside 3.46]{kollar}}]\label{p-ex12.1}
In this example, we will freely use Koll\'ar's notation in \cite[Aside 3.46]{kollar}. 
In \cite[Aside 3.46]{kollar}, 
we assume that $C$ is an elliptic curve. 
Then the surface $S$ constructed in \cite[Aside 3.46]{kollar} 
is a complete non-projective 
algebraic surface with two simple 
elliptic singularities. 
In particular, $S$ is Gorenstein and log canonical. 
Let $C'$ be the strict transform of $\{1\}\times C$ on $S$. 
Then we have $\Pic (S)=\mathbb Z\mathcal O_S(C')$. 
We can directly check that 
$\pi_1(S)=\{1\}$, that is, $S$ is simply connected, 
$K_S\sim 0$, 
$H^1(S, \mathcal O_S)=0$, and $H^2(S, \mathcal O_S)=\mathbb C$. 
Therefore, $S$ is a log canonical 
Calabi--Yau algebraic surface. 
Of course, we have $\kappa (S, K_S)=0$. 
Let $F$ be a general fiber of the second projection 
$\mathbb P^1\times C\to C$ and let $F'$ 
be the strict transform of $F$ on $S$. 
Then $\NE(S)=\mathbb R_{\geq 0} [F']$ holds. 
We note that $\mathcal O_S(C')$ is positive 
on $\NE(S)\setminus \{0\}$. 
However, $C'$ is nef but is not ample. 
This means that 
Kleiman's ampleness criterion does not hold for $S$. 
\end{ex}

Let us prepare an easy lemma. 

\begin{lem}\label{p-lem12.2} 
Let $C$ be a smooth projective curve and let $\mathcal A$ 
be an ample 
line bundle 
on $C$. 
We consider $\pi\colon X:=\mathbb 
P_C(\mathcal O_C\oplus \mathcal 
A)\to C$. Then $\mathcal O_X(1)\simeq \mathcal O_X(C_+)$ is 
semi-ample, where $C_+:=\mathbb P_C(\mathcal A)$ is the positive 
section of $\pi$, and the complete linear system 
$|\mathcal O_X(m)|$ only contracts 
the negative section $C_-:=\mathbb P_C(\mathcal O_C)$ of 
$\pi$ to a point for some sufficiently large positive integer $m$. 
\end{lem}

\begin{proof}
We can easily check that $\mathcal O_X(1)$ is semi-ample (see 
\cite[Lemma 2.3.2]{lazarsfeld}). 
We note that $C_-\cdot \mathcal O_X(1)=0$ and 
$D\cdot \mathcal O_X(1)>0$ for every irreducible 
curve $D$ on $X$ with $D\ne C_-$. 
Therefore, the complete linear system $|\mathcal O_X(m)|$ 
contracts $C_-$ only. 
\end{proof}

Let us construct complete non-projective normal algebraic surfaces $S$ with 
$\Pic(S)=\{0\}$. 
The following construction was suggested by Kento Fujita. 

\begin{ex}\label{p-ex12.3}
Let $C$ be a smooth projective curve of genus $\geq 1$ and 
let $\mathcal L=\mathcal O_C(L)$ be a non-torsion 
element of $\Pic^0(C)$. 
We considr $\pi\colon V:=\mathbb P_C(\mathcal O_C\oplus 
\mathcal L)\to C$. 
Let $C_1$ (resp.~$C_2$) be the section of $\pi$ corresponding to 
$\mathcal O_C\oplus \mathcal L\to \mathcal O_C\to 0$ (resp.~$\mathcal O_C
\oplus \mathcal L\to \mathcal L\to 0$). 
We note that $C_2\sim \pi^*L+C_1$ holds. 
We take an arbitrary point $P\in C$ and blow up 
$P_1$ and $P_2$, where $P_i:=
\pi^{-1}(P)\cap C_i$ for $i=1, 2$, to get $p\colon W\to V$. 
Let $C_i'$ be the strict transform of $C_i$ on $W$ for $i=1, 2$ 
and let $\ell$ be the strict transform of $\pi^{-1}(P)$ on $W$. 
Let $E_i$ denote the $(-1)$-curve on $W$ with $p(E_i)=P_i$ for $i=1, 2$. 
We put  
$$
U_1:=W\setminus (C_2'\cup E_2\cup \ell), \quad 
U_2:=W\setminus (C_1'\cup E_1\cup \ell), \quad 
\text{and} 
\quad 
U_0:=W\setminus (C_1'\cup C_2'). 
$$
We note that $C_i'\subset U_i$ for $i=1,2$ and 
$\ell \subset U_0$ by construction. 
Then we can realize $U_i$ as a Zariski open subset of $\mathbb P_C(\mathcal O_C\oplus \mathcal A_i)$ with 
$\deg \mathcal A_i=1$ such that $C_i'$ corresponds to 
the negative section $C_-$ on $\mathbb P_C(\mathcal O_C\oplus \mathcal A_i)$ 
for $i=1, 2$. 
By Lemma \ref{p-lem12.2}, 
we can construct a projective birational morphism 
$U_i\to S_i$ onto a normal quasi-projective surface $S_i$ such that the 
exceptional locus is $C_i'$ for $i=1, 2$. 
Since $\ell$ is a $(-2)$-curve on a smooth 
projective surface $W$, 
we can construct a projective birational morphism 
$U_0\to S_0$ onto a normal 
quasi-projective surface $S_0$ such that the exceptional 
locus is $\ell$ and that 
$\ell$ is contracted to an $A_1$ singularity. 
Now $S_0$, $S_1$, and $S_2$ can be glued together 
to get a birational contraction morphism 
$q\colon W\to S$ onto a complete normal 
algebraic surface $S$ which only contracts $C_1'$, $C_2'$, and $\ell$. 
\begin{claim}\label{p-ex12.3-claim}
$\Pic(S)=\{0\}$ holds. 
\end{claim}
\begin{proof}[Proof of Claim \ref{p-ex12.3-claim}]
We take an arbitrary Cartier divisor $D$ on $S$. 
We put $\overline D:=q^*D$ and $D^\dag:=p_*\overline D$. 
Then we can write
$D^\dag\sim \pi^*M+kC_1$ for some 
integer $k$ and some Cartier divisor $M$ on $C$ because 
$V$ is $\mathbb P_C(\mathcal O_C\oplus \mathcal L)$. 
We note that we can write 
$$
p^*D^\dag =\overline D+m_1E_1+m_2E_2 
$$ 
for some integers $m_1$ and $m_2$. 
By construction, $\mathcal O_W(\overline D)|_{C_i'}$ is trivial for $i=1,2$. 
Therefore, 
we have 
$$
\mathcal O_C(M)\simeq \mathcal O_V(D^\dag)|_{C_2}
\simeq \mathcal O_W(p^*D^\dag)|_{C_2'}\simeq 
\mathcal O_W(\overline D+m_1E_1+m_2E_2)|_{C_2'}\simeq 
\mathcal O_C(m_2P)
$$
and 
$$
\mathcal O_C(M-kL)\simeq \mathcal O_V(D^\dag)|_{C_1}
\simeq \mathcal O_W(p^*D^\dag)|_{C_1'}\simeq 
\mathcal O_W(\overline D+m_1E_1+m_2E_2)|_{C_1'}\simeq 
\mathcal O_C(m_1P). 
$$
This implies that $\mathcal O_C(kL)\simeq \mathcal O_C((m_2-m_1)P)$. 
Since $\deg L=0$, we have $m_1=m_2$. 
By assumption, $\mathcal L=\mathcal O_C(L)$ is a non-torsion 
element of $\Pic^0(C)$. 
Thus, we get $k=0$. 
By construction again, we have $\overline D\cdot \ell=0$. 
Therefore, 
$$
0=k=p^*D^\dag\cdot \ell =(\overline D+m_1E_1+m_2E_2)\cdot \ell=m_1+m_2
$$ 
holds. 
Since we have already known that $m_1=m_2$ holds, 
we obtain that $m_1=m_2=0$. 
Thus, we have $\overline D=p^*D^\dag\sim p^*\pi^*M$. 
Since $\mathcal O_W(\overline D)|_{C_i'}$ is trivial for $i=1, 2$, 
we see that $\mathcal O_C(M)$ is trivial. 
Therefore, we obtain $\overline D\sim 0$. 
So we get $D\sim 0$. 
This means that $\Pic (S)=\{0\}$. 
\end{proof}
By Claim \ref{p-ex12.3-claim}, 
we see that $S$ is obviously non-projective. 
By construction, we can see that $\pi_1(S)=\{1\}$, that is, 
$S$ is simply connected. 

From now on, we assume that $C$ is an elliptic curve. 
Then $S$ has two simple elliptic singularities and one 
$A_1$ singularity. 
Moreover, $S$ is Gorenstein and $K_S\sim 0$. 
We note that $K_V+C_1+C_2 \sim 0$ and 
$K_W+C_1'+C_2'\sim 0$ by construction.  
We can easily check that 
$H^1(S, \mathcal O_S)=0$ and $H^2(S, \mathcal O_S)=\mathbb C$ by 
using the Leray spectral sequence. 
Therefore, $S$ is a complete non-projective log canonical 
Calabi--Yau algebraic surface with $\Pic(S)=\{0\}$. 
\end{ex}

From now on, by taking blow-ups of $S$ in Example 
\ref{p-ex12.3}, 
we construct complete non-projective 
algebraic surfaces with large Picard number. 

\begin{ex}\label{p-ex12.4}
Let $S$ be the surface constructed in Example \ref{p-ex12.3}, which 
is birationally equivalent to $\mathbb P_C(\mathcal O_C\oplus \mathcal L)$. 
As we saw above, 
$\Pic (S)=\{0\}$ holds. 
We take the minimal resolution $\mu\colon \widetilde S\to S$ of the 
unique $A_1$ singularity of $S$. 
Since $S$ is simply connected, so is $\widetilde S$. 
Let $E$ denote the exceptional curve of $\mu$. 
Of course, $E$ is the strict transform of $\ell$ 
in Example \ref{p-ex12.3} and is a $(-2)$-curve on $\widetilde S$. 
\begin{claim}\label{p-ex12.4-claim} 
$\Pic (\widetilde S)=\mathbb Z\mathcal O_{\widetilde S}(E)$ holds. 
\end{claim}
\begin{proof}[Proof of 
Claim \ref{p-ex12.4-claim}] 
By construction, we see that $E\simeq \mathbb P^1$ and that 
$-E$ and $-(K_{\widetilde S}+E)$ are both $\mu$-ample. 
Thus, 
by Theorem \ref{f-thm3.11} (see also Remark \ref{f-rem11.4}), 
we have $\Pic (\widetilde S)\otimes \mathbb Q=\mathbb Q 
\mathcal O_{\widetilde S}(E)$ since $\Pic (S)=\{0\}$. 
Let $\mathcal N$ be a torsion element of $\Pic(\widetilde S)$. 
Then we have $\mathcal N\cdot E=0$. 
Therefore, by Theorem \ref{f-thm3.11} again, 
$\mathcal N$ is trivial. 
This means that $\Pic(\widetilde S)$ is torsion-free. 
Thus, we can write $\Pic(\widetilde S)=\mathbb Z
\mathcal M$ for some line bundle $\mathcal M$ on $\widetilde S$. 
Therefore, there exists some integer $a$ such that 
$\mathcal O_{\widetilde S}(E)\simeq 
\mathcal M^{\otimes a}$. 
Since $E^2=-2$, we have $aE\cdot \mathcal M=-2$. 
Note that $E\cdot \mathcal M$ is an integer. 
If $a=\pm 2$, then $-2=E^2=a^2\mathcal M^2=4\mathcal M^2$. 
This is a contradiction because 
$\mathcal M^2$ is an integer. 
Thus, we get $a=\pm 1$. 
This means that $\Pic 
(\widetilde S)=\mathbb Z \mathcal O_{\widetilde S} (E)$. 
\end{proof}
By Claim \ref{p-ex12.4-claim}, 
we have $N^1(\widetilde S)=\mathbb R$ and 
$\NE(\widetilde S)=N_1(\widetilde S) =\mathbb R$. 
We note that $E^2=-2$ and that 
there exists a curve $G$ on $\widetilde S$ such that 
$G\cdot E>0$. 

We further assume that $C$ is an elliptic curve. 
Then $K_{\widetilde S}\sim 0$ and $\widetilde S$ has only 
two simple elliptic singularities. 
Therefore, $\widetilde S$ is a complete non-projective 
log canonical Calabi--Yau algebraic surface with 
$\NE(\widetilde S)=\mathbb R$. 
\end{ex}

\begin{ex}\label{p-ex12.5}
Let $\widetilde S$ be the surface constructed in Example 
\ref{p-ex12.4}. 
We take finitely many smooth points 
$Q_1, Q_2, \ldots, Q_{\rho-1}$ of $\widetilde S$ with 
$\rho \geq 2$ such that $Q_i\ne Q_j$ for $i\ne j$ and 
$Q_i\not\in E$ for every $i$. 
We blow up $Q_1, Q_2, \ldots, Q_{\rho-1}$ to get 
$\sigma \colon \overline S\to \widetilde S$. 
Let $B_i$ denote the $(-1)$-curve on $\overline S$ with 
$\sigma(B_i)=Q_i$ for 
every $i$. 
Then, by Theorem \ref{f-thm3.11} (see also Remark 
\ref{f-rem11.4}), 
we can easily check 
that 
$$
\Pic (\overline S)=\mathbb Z\mathcal O_{\overline S} (\sigma ^*E) 
\oplus \mathbb Z \mathcal O_{\overline S} (B_1)\oplus 
\cdots \oplus \mathbb Z\mathcal O_{\overline S}(B_{\rho-1}). 
$$ 

\begin{claim}\label{p-ex12.5-claim} 
$\NE(\overline S)=N_1(\overline S)=\mathbb R^\rho$ holds. 
\end{claim} 
\begin{proof}[Proof of Claim \ref{p-ex12.5-claim}]
Let $D$ be a nef $\mathbb R$-Cartier 
$\mathbb R$-divisor on $\overline S$. 
It is sufficient to 
prove that $D$ is numerically trivial. 
By the above description of 
$\Pic (\overline S)$, 
we can assume that $D=b_0 \sigma^*E+\sum _{i=1}^{\rho-1} b_i B_i$, 
where $b_i\in \mathbb R$ for every $i$. 
Since $D\cdot B_i\geq 0$ for every $i$ and 
$D\cdot \sigma^*E\geq 0$, 
$b_i\leq 0$ holds for every $i$. 
We assume that $D\ne 0$. 
Then we can take an irreducible curve $G$ on $\overline S$ such that 
$G\not\subset \Supp D$ and that $G\cap \Supp D\ne \emptyset$. 
This is because the smooth locus of $\overline S$ 
is a quasi-projective open subset of $\overline S$ and contains 
$\sigma^*E$ and $B_i$ for every $i$. 
Thus we get 
$D\cdot G<0$. 
This is a contradiction. 
This means that $D=0$. 
Hence we have 
$\NE(\overline S)=N_1(\overline S)=\mathbb R^\rho$. 
\end{proof}

If $C$ is an elliptic curve, then 
$\overline S$ is a complete non-projective 
log canonical algebraic surface 
with $\NE(\overline S)=N_1(\overline S)=\mathbb R^\rho$. 
By construction, we see that $K_{\overline S}=\sum _{i=1}^{\rho-1} B_i$. 
We can apply the minimal model program established in Theorem 
\ref{f-thm1.5} to a complete non-projective 
log canonical algebraic surface $\overline S$. 
Then every $B_i$ is contracted to a smooth point and we finally get a 
good minimal model $\widetilde S$, 
which is a surface with $K_{\widetilde S}\sim 0$ 
constructed in Example \ref{p-ex12.4}. 
\end{ex}

The following commutative diagram may help the reader understand 
the constructions in Examples \ref{p-ex12.3}, 
\ref{p-ex12.4}, and \ref{p-ex12.5}. 

$$
\xymatrix{
&W\ar[dl]_-p\ar[dr]_-q\ar[rr]& &\widetilde S \ar[dl]^-\mu& \overline S
\ar[l]_-\sigma\\
V\ar[d]_-\pi && S && \\ 
C &&&& 
}
$$

The reader can find various examples of complete 
non-projective toric threefolds $X$ with $\Pic(X)=\{0\}$,
$\NE(X)=\mathbb R_{\geq 0}$, or 
$\NE(X)=N_1(X)$ in \cite{fujino-kleiman} and 
\cite{fujino-payne}. 

\medskip 

Finally, we make a remark on complete non-projective 
algebraic surfaces defined over 
an algebraically closed filed $k$ with $k\ne \mathbb C$. 

\begin{rem}\label{p-rem12.6}
We note that every complete algebraic surface defined 
over $\overline {\mathbb F}_p$ is always $\mathbb Q$-factorial 
(see \cite[Theorem 4.5]{tanaka}). 
So it automatically becomes projective (see 
\cite[Lemma 2.2]{fujino-surfaces}). 
By the following lemma (see Lemma \ref{p-lem12.7} below), the 
constructions in this section and Koll\'ar's construction in 
\cite[Aside 3.46]{kollar} 
can work for algebraic surfaces defined over an 
algebraically closed field $k$ such $k\ne  
\overline {\mathbb F}_p$ for every prime number $p$ 
with some suitable modifications. 
We note that Theorem \ref{f-thm3.11} 
holds true for algebraic surfaces 
defined over any algebraically closed field 
(see Remark \ref{f-rem11.4}). Thus, 
we can construct complete non-projective algebraic surfaces 
over an 
algebraically closed field $k$ such that 
$k \ne \overline {\mathbb F}_p$ 
for every prime number $p$. 
\end{rem}

\begin{lem}\label{p-lem12.7}
Let $C$ be a smooth projective 
curve defined over an algebraically closed field 
$k$ whose genus $g(C)$ is positive. 
Let $P$ be an arbitrary closed point of $C$. 
Assume that $k\ne \overline {\mathbb F}_p$ for every 
prime number $p$. 
Then there exists $Q\in C$ such that $\mathcal O_C(Q-P)$ is a non-torsion 
element of $\Pic^0(C)$. 
\end{lem}

\begin{proof}
By $g(C)\geq 1$ and $k\ne \overline {\mathbb F}_p$, 
we can take a non-torsion element $\mathcal L$ of $\Pic^0(C)$ 
(see \cite[Fact 2.3]{tanaka}). 
We take a large positive integer $m$. 
Then $\mathcal L\otimes \mathcal O_C(mP)$ is very ample. 
We consider a general member $Q_1+\cdots +Q_m$ 
of $|\mathcal L\otimes \mathcal O_C(mP)|$. 
Then we have 
$\mathcal L\simeq \mathcal O_C\left((Q_1-P)+\cdots +(Q_m-P)\right)$. 
Therefore, there exists some $i_0$ such that 
$\mathcal O_C(Q_{i_0}-P)$ is a non-torsion 
element of $\Pic^0(C)$. 
\end{proof}


\end{document}